\documentclass[12pt,letterpaper,english]{amsart}
\usepackage{lmodern}
\usepackage{helvet}

\usepackage[T1]{fontenc}
\usepackage[latin9]{inputenc}
\usepackage{mathrsfs}
\usepackage{amsthm}
\usepackage{amstext}
\usepackage{amssymb}
\usepackage{tikz}
\usepackage{esint}
\usepackage{graphicx}
\usepackage{cancel}
\usepackage{enumitem}
\DeclareFontFamily{OT1}{pzc}{}
\DeclareFontShape{OT1}{pzc}{m}{it}{<-> s * [1.10] pzcmi7t}{}
\DeclareMathAlphabet{\mathpzc}{OT1}{pzc}{m}{it}
\usepackage[bbgreekl]{mathbbol}
\DeclareSymbolFontAlphabet{\mathbb}{AMSb}
\DeclareSymbolFontAlphabet{\mathbbl}{bbold}

\makeatletter

\pdfpageheight\paperheight
\pdfpagewidth\paperwidth

\newcommand{\noun}[1]{\textsc{#1}}

\numberwithin{equation}{section}
\numberwithin{figure}{section}
\usepackage{enumitem}		
\theoremstyle{plain}
\newtheorem{thm}{\protect\theoremname}[section]
\theoremstyle{remark}
\newtheorem{rem}[thm]{\protect\remarkname}
\theoremstyle{definition}
\newtheorem{defn}[thm]{\protect\definitionname}
\theoremstyle{definition}
\newtheorem{example}[thm]{\protect\examplename}
\theoremstyle{plain}
\newtheorem{cor}[thm]{\protect\corollaryname}
\theoremstyle{plain}
\newtheorem{prop}[thm]{Proposition}
\theoremstyle{plain}
\newtheorem{lem}[thm]{Lemma}

\usepackage{amsfonts}
\usepackage{amssymb}
\usepackage[all]{xy}
\usepackage{array}
\usepackage{lmodern}
\usepackage[T1]{fontenc}
\usepackage{bm}

\def\d{\partial}
\def\uo{\underline{0}}

\def\ux{\underline{x}}
\def\QQ{\mathbb{Q}}
\def\RR{\mathbb{R}}
\def\CC{\mathbb{C}}

\def\ZZ{\mathbb{Z}}
\def\PP{\mathbb{P}}

\def\DD{\mathbb{D}}
\def\MM{\mathbb{M}}
\def\GG{\mathbb{G}}

\def\VV{\mathbb{V}}

\def\F{\mathcal{F}}

\def\V{\mathcal{V}}

\def\X{\mathcal{X}}

\def\M{\mathcal{M}}
\def\H{\mathcal{H}}
\def\HH{\mathbb{H}}

\def\KK{\mathbb{K}}
\def\D{\mathcal{D}}

\def\vf{\varphi}

\def\k12{\mathcal{K}_{\lambda_1,\lambda_2}}
\def\tk12{\tilde{\mathcal{K}}_{\lambda_1,\lambda_2}}
\def\ck12{\check{\mathcal{K}}_{\lambda_1,\lambda_2}}

\def\bx{\zeta}

\def\ay{\mathbf{i}}

\def\ve{\varepsilon}
\def\e{\epsilon}
\def\cx{\mathcal{X}}

\def\IH{\mathrm{IH}}
\def\co{\mathcal{O}}
\def\ord{\mathrm{ord}}
\def\rk{\mathrm{rk}}
\def\F{\mathcal{F}}
\def\fb{\F^{\bullet}}

\def\gr{\mathrm{Gr}}
\def\cl{\mathcal{L}}
\def\clv{\cl^{\dagger}}
\def\B{\mathcal{B}}
\def\r{\mathpzc{r}}
\def\q{\mathpzc{q}}
\def\mc{\mathpzc{m}_c}
\def\mm{\mathpzc{m}}
\def\fa{\mathfrak{a}}
\def\vc{\ve^{\circ}}
\def\K{\mathcal{K}}
\def\E{\mathcal{E}}
\def\KK{\mathbb{K}}
\def\EE{\mathbb{E}}
\def\rt{\tilde{\mathcal{R}}}
\def\el{\mathpzc{L}}
\def\nf{\mathpzc{v}}
\def\v{\mathsf{v}}
\def\vt{\vartheta}
\def\fu{\mathfrak{U}}
\def\fux{\fu^{\times}}
\def\fuc{\fu_c}
\def\fuo{\fu_0}
\def\fucx{\fuc^{\times}}
\def\fuox{\fuo^{\times}}

\def\NP{\mathbbl{\Delta}}

\theoremstyle{definition}

\theoremstyle{definition}

\theoremstyle{definition}
\newtheorem*{thx}{Acknowledgments}

\makeatother

\usepackage{babel}
  \providecommand{\corollaryname}{Corollary}
  \providecommand{\definitionname}{Definition}
  \providecommand{\remarkname}{Remark}
\providecommand{\theoremname}{Theorem}
 \providecommand{\examplename}{Example}

\begin{document}

\title[Unipotent extensions and differential equations]{Unipotent extensions and differential equations (after Bloch-Vlasenko)}

\author{Matt Kerr}

\begin{abstract}
S. Bloch and M. Vlasenko recently introduced a theory of \emph{motivic Gamma functions}, given by periods of the Mellin transform of a geometric variation of Hodge structure, which they tie to the monodromy and asymptotic behavior of certain unipotent extensions of the variation.  Here we further examine these Gamma functions and the related \emph{Ap\'ery and Frobenius invariants} of a VHS, and establish a relationship to motivic cohomology and solutions to inhomogeneous Picard-Fuchs equations.
\end{abstract}
\maketitle

\section{Introduction}\label{SI}

The Frobenius method for solving linear ODEs in the neighborhood of a regular singular point (see for example \cite{Yo}) goes all the way back to \cite{Fr}.  The significance of the resulting basis of solutions in Hodge theory and mirror symmetry has recently been elevated by two seminal papers.  In their proof of the Gamma Conjecture for rank-one Fano threefolds \cite{GZ}, Golyshev and Zagier studied the Frobenius solutions for the regularized quantum differential equations of these Fanos, using their monodromy to define constants $\kappa_0,\kappa_1,\kappa_2,\kappa_3$ and matching those to the coefficients of the Gamma-class of each Fano; they also obtain a natural extension of the $\{\kappa_j\}$ to a (more mysterious) infinite sequence.  Subsequently, Bloch and Vlasenko \cite{BV} generalized these \emph{Frobenius constants} to a broader class of Picard-Fuchs equations, and gave them a new interpretation, as periods of the limiting mixed Hodge structure of the underlying variation and its unipotent extensions.  They also showed that the generating series $\kappa(s):=\sum_{j\geq 0}\kappa_j s^j$ is essentially a \emph{motivic Gamma function}, that is, a period of the Mellin transform (as defined by \cite{LS}) of the underlying $\mathcal{D}$-module.

In this paper, we study the properties of $\kappa(s)$ for a particular class of Picard-Fuchs equations, attached to polarized variations of Hodge structure over a Zariski open set $U\subset \PP^1$ with all Hodge numbers equal to $1$ (and a few other properties detailed below).  Our first main goal is simply to give a streamlined presentation of the main results of Bloch and Vlasenko in this case, making occasional technical improvements (Theorems \ref{T5.2} and \ref{T7.1}), and using the polarization to make the ``$\Gamma=\kappa$'' result more explicit (Theorem \ref{T6.1}).  We also highlight how their work can be used to compute LMHSs (Example \ref{E5.1}) and produce a \emph{limiting motive} in the hypergeometric case (Remark \ref{R6.4}).  Our second goal is to interpret certain features of $\kappa$ in terms of motivic cohomology and admissible normal functions.  For instance, if the variation has weight $n$ (and rank $n+1$), then $\kappa_{n+1}$ is the first Frobenius number not related to its LMHS; in Theorem \ref{T7.2}, we obtain a motivic interpretation of the ``first unipotent extension'' of \cite[\S5]{BV}, and hence of this number, confirming a speculation in the closing pages of [loc. cit.].  In \S\ref{S8}, we investigate the values of $\kappa$ at positive integers, which we term \emph{Ap\'ery constants}.  After characterizing them as special values of solutions to inhomogeneous equations (Theorem \ref{T8.1}), we interpret them in some cases as regulators of higher cycles (Theorems \ref{T8.2} and \ref{T8.3}).

In the remainder of this Introduction, we offer a brief mathematical \emph{dramatis personae} for the reader's reference (see the next page).

To set the scene:\footnote{For simplicity, we impose assumptions largely avoided in the text:  strong conifold monodromy at $c$ (which goes a bit beyond $\rk(T_c-I)=1$, see \S\ref{S3}), self-adjointness of $L$ (see \S\S\ref{S4}-\ref{S5}), and $\M$ arising from a family defined over $\bar{\QQ}$.} let $\Sigma=\{0,c,\ldots,\infty\}\subset\PP^1$ be finite, with $|c|<|c'|$ for all $c'\in \Sigma\setminus\{0,c\}$.  Let $\mathbb{D}$ be an open disk centered about $0$ with $\mathbb{D}\cap \Sigma=\{0,c\}$; and, writing $U:=\PP^1\setminus\Sigma$, fix $p\in\mathbb{D}\cap U$.  Consider a $\bar{\QQ}$-motivic, polarized $\QQ$-VHS $\M$ on $U$, of weight $n$ with Hodge numbers $h^{p,n-p}=1$ ($0\leq p\leq n$).\footnote{Since $\M$ has a rational polarization $Q$, it is self-dual, so that the dual of $D^mL$ is $LD^m$ below. We still find it useful however to formally distinguish $\M$ and $\M^{\vee}$ for some purposes: we use $Q(\cdot,\cdot)$ to denote the pairing on $\M$ (or $\M^{\vee}$), and $\langle\cdot,\cdot\rangle$ for the pairing of $\M$ and $\M^{\vee}$; see \S\ref{S3}.} Suppose the underlying local system has maximal unipotent monodromy at $t=0$, and strong conifold monodromy at $t=c$, represented by $T_0,T_c\in \mathrm{Aut}(\MM_{\QQ,p})$ (with $N_0:=\log(T_0)$); assume in addition that $\ker(T_0-I)\cap \ker(T_c-I)=\{0\}$.  Write $\gamma_0,\gamma_c\in\pi_1(\mathbb{D}\cap U)$ for loops based at $p$ winding once about $0,c$.

Fixing $\ve_0\in (\MM_{\QQ,p}^{\vee})^{T_0}$, there is a unique basis $\{\ve_0,\ldots,\ve_n\}\subset \MM_{\QQ,p}^{\vee}$ such that $N_0\ve_j=\ve_{j-1}$ and $(T_c-I)\ve_j=0$ for $j>0$.  Set $\delta:=(T_c-I)\ve_0\in\MM_{\QQ,p}^{\vee}$ and put $\mathsf{Q}_0:=Q(\ve_0,\ve_n)$, $\mathsf{Q}_c:=Q(\ve_0,\delta)$ (both in $\QQ^{\times}$).  Choose $\mu\in H^0(\PP^1,\F^n_e\M_e)$ the (unique) section of the canonically extended Hodge line which is nowhere zero on $\PP^1\setminus\{\infty\}$, and normalized so that the ``fundamental period'' $\phi_0(t):=\langle \ve_0,\mu\rangle =:\sum_{k\geq 0}a_k t^k$ has $a_0=1$.  Write $\psi(t):=\langle\delta,\mu\rangle$ and $\e_j(t):=\langle \ve_j,\mu\rangle$ for other periods, and $\e_j^{\text{an}}(t)$ for the analytic (at $0$) part of $\epsilon_j(t)$.  The left-hand column of the period matrix of the LMHS of $\M$ at $0$ is given by $(2\pi\ay)^j\e_j^{\text{an}}(0)$, $0\leq j\leq n$.

\subsection*{Picard-Fuchs}
$L:=\sum_{j=0}^d t^j P_j(D)\in \CC[t,D]$ is the minimal operator with $\nabla_L\mu=0$ (hence $L\e_j=0=L\psi$).  It has order $n+1$ and degree $d$.

\subsection*{Conifold Gamma}
$\Gamma_c(s):=\sum_{k=0}^{n+1}(-1)^{n+1-k}\binom{n+1}{k}e^{2\pi\ay ks}\int_{\gamma_0^{-k}}\psi(t)t^{s}\tfrac{dt}{t}+(e^{2\pi\ay s}-1)^{n+1}\int_{\gamma_c}\e_0(t)t^{s}\tfrac{dt}{t}$ is entire, with $\sum_{j=0}^d P_j(-s-j)\Gamma_c(s+j)=0$, and $\Gamma_c(-m)=(-1)^{n+1}\tfrac{\mathsf{Q}_c}{\mathsf{Q}_0}(2\pi\ay)a_k$ for $k\in\ZZ_{\geq 0}$.

\subsection*{Frobenius periods}
$\Phi(s,t)=\sum_{\ell\geq 0}\phi_{\ell}(t)s^{\ell}$ is uniquely defined by $L\Phi=s^{n+1}t^s$ and $T_0\Phi=e^{2\pi\ay s}\Phi$.  Write $\phi_{\ell}(t)=:\sum_{b=0}^{\ell}\tfrac{1}{b!}\log^{b}(t)\phi_{\ell-b}^{\text{an}}(t)$ and $\phi_{\ell}^{\text{an}}(t)=:\sum_{k\geq 0}a_k^{(\ell)}t^k$ (noting that $a_k^{(0)}=a_k$, and $a_0^{(\ell>0)}=0$).  Then $A_k(s):=\sum_{\ell\geq 0}a^{(\ell)}_k s^{\ell}$ satisfies $\Phi(s,t)=\sum_{k\geq 0}A_k(s) t^{s+k}$.

\subsection*{Kappa series}
$(T_c-I)\Phi(s,t)=:\kappa(s)\psi(t)$, with $\kappa(s)=:\sum_{j=0}^{\infty}\kappa_j s^j$ and $\kappa(s)^{-1}=:\sum_{j=0}^{\infty}\alpha_j s^j$.  We have $\kappa_0=\alpha_0=1$ and $\alpha_j=(2\pi\ay)^j \e_j^{\text{an}}(0)$ for $0\leq j\leq n$.  Moreover, we have the asymptotic formulas $\kappa(s)=c^s\cdot \lim_{k\to \infty}\tfrac{A_k(s)}{a_k}$ and $\kappa_j=\sum_{j=0}^{\ell}\tfrac{1}{j!}\log^j(c)\cdot \lim_{k\to \infty}\tfrac{a_k^{(\ell-j)}}{a_k}$.  The basic relation between Gamma and kappa is $\Gamma_c(s)=\tfrac{(1-e^{2\pi\ay s})^{n+1}\mathsf{Q}_c}{(2\pi\ay)^n s^{n+1}\mathsf{Q}_0}\kappa(s)$ in this self-dual setting.  At $s\sim -k$ we therefore have $\kappa(s)\sim \tfrac{(-k)^{n+1}}{(s+k)^{n+1}}a_k$.
 
\subsection*{Unipotent extension}
Fix $m\in\ZZ_{>0}$.  There is a unique extension $0\to\mathcal{K}_m\to \mathcal{E}_m \to \M\to 0$ of admissible $\QQ$-VMHS on $\Delta_0^{\times}$ (a small punctured disk about $0$) with underlying $\QQ$-local system $\mathbb{E}_m$ extending to $\mathbb{D}\cap U$, underlying $\mathcal{D}$-module $\mathcal{D}/\mathcal{D}D^m L$, and with $\mathcal{K}_m$ of rank $m$ with Hodge numbers $h^{-m,-m}=\cdots=h^{-1,-1}=1$.  The coefficients $\{\alpha_j\}_{0\leq j\leq n+m}$ of $\kappa(s)^{-1}$ yield the left-hand column of the period matrix of the LMHS of $\mathcal{E}_m$ at $0$.

\underline{Key Example 1}: if $\varphi\in \bar{\QQ}[x_1^{\pm1},\ldots,x_{n+1}^{\pm1}]$ is reflexive and tempered, and $f=\tfrac{1}{\vf}\colon\X\to\PP^1$ the resulting CY-$n$-fold family (with $\M\subseteq R^n f_*\QQ$ as above), the {\bf box extension} --- arising from fiberwise restriction of (roughly) the symbol $\{x_1,\ldots,x_{n+1}\}\in K^M_{n+1}(\bar{\QQ}(\cx))$ --- is $\mathcal{E}^{\vee}_1(1)$.

\subsection*{Inhomogeneous equations}
For any $\ell\in\ZZ_{>0}$, let $V^{[\ell]}(t)$ denote the unique solution to $L(\cdot)=-t^{\ell}$ analytic on $\mathbb{D}$; then $\kappa(\ell)=\ell^{n+1}V^{[\ell]}(0)$.  Each embedding of a Tate object $\QQ(-a)\hookrightarrow \IH^1(\PP^1\setminus \{\infty \},\MM)$ produces an admissible extension $0\to \M\to \mathcal{V}_{\mu}\to \QQ(-a)\to 0$ with higher normal function $V_{\mu}(t)$ of this type for $\ell\leq d$.  

\underline{Key Example 2}: if $d=2$, then $\IH^1(\PP^1\setminus\{0\},\MM)\cong \QQ(-a)$ for some $\tfrac{n+1}{2}\leq a\leq n+1$, and the resulting {\bf higher normal function} $V_{\mu}$ satisfies $LV_{\mu}=-\beta t$ for some $\beta \in \CC^{\times}$, and $\kappa(1)=\beta^{-1}V_{\mu}(0)$.  Of course, $\M$ usually arises from a family $\X$ defined over $\bar{\QQ}$, and then $\beta\in\bar{\QQ}$.

\subsection*{Summary}
We record the basic properties of the kappa series, which is really a \emph{meromorphic function} on $\CC$ with poles at negative integers:

(1) At $s=-k\in\ZZ_{\leq 0}$, the leading term in the Laurent expansion of $\kappa(s)$ is $\kappa^*(-k)=(-k)^{n+1}a_k$.  Here $\{a_k\}$ are the coefficients of the unique holomorphic period of $\M$ on $\Delta_0$; in Key Example 1, $a_k=[\varphi^k]_{\uo}$ are constants in the powers of the Laurent polynomial.

(2) At $s=0$, the power series coefficients of $\kappa$ (more precisely, of $\kappa^{-1}$) compute the LMHS of $\M$ --- and, more generally, of $\mathcal{E}_m$ --- at $t=0$.  These are the numbers arising in \cite{GZ}. In Key Example 1 (with $\varphi$ the Minkowski polynomial mirror to a Fano $X^{\circ}$), by the Gamma Conjecture they should match the coefficients of powers of $c_1$ in the regularized $\hat{\Gamma}$-class of $X^{\circ}$ --- and, more generally, of its ``progenitors'' (see \cite{Go2}).  (In the case of $\mathcal{E}_1$, $\kappa_{n+1}$ is related to the LMHS of the box extension at $t=0$; but this is \emph{not} the special value of the corresponding higher normal function, which blows up at $0$ in any normalization --- the \emph{extension of VMHS} cannot be specialized there.)

(3) At $s=k\in \ZZ_{>0}$, the values $\kappa(k)$ reflect the value at $0$ of the unique solution to the inhomogeneous equation $L(\cdot)=-t^{k}$ analytic on the big disk $\mathbb{D}$.  When certain hypotheses are satisfied,\footnote{namely, that $\IH^1(\mathbb{A}^1,\MM)$ be split Hodge-Tate (or at least have ``enough'' Hodge classes), as well as the Beilinson-Hodge Conjecture for the family $\X$ underlying $\M$.} for small values of $k$ these will be special values of higher normal functions arising from motivic cohomology classes on $\cx\setminus X_{\infty}$.  These are the numbers that arise in \cite{Go}, and are expected to be the correct B-model interpretation of Ap\'ery constants of homogeneous varieties tabulated in \cite{Ga}.  Moreover, they are the numbers which arise in the ``spirit of Ap\'ery'' (in taking a linear combination of two exponentially increasing solutions to a recurrence that then dies exponentially).

In light of (2) and (3), it seems reasonable to call the $\{\kappa(k)\}$ Ap\'ery numbers and the $\{\kappa_j\}$ Frobenius numbers.  These are global arithmetic invariants of the VHS $\M$.\vspace{2mm}

\noindent Some mundane notational conventions:  we write $\bm{\delta}_{ij}$ for the Kronecker delta, $\ay:=\sqrt{-1}$, $\el(t):=\tfrac{\log(t)}{2\pi\ay}$, and $D:=t\tfrac{d}{dt}$.

\begin{thx}
The author thanks S. Bloch, V. Golyshev, and A. Klemm for many helpful discussions.  This work was partially supported by Simons Collaboration Grant 634268.
\end{thx}

\section{Periods of connections}\label{S1}
Fix a coordinate $t$ on $\PP^1$.  We work in the setting of algebraic connections on $U:=\PP^1 \setminus \Sigma$, where $\Sigma$ is a set of at least three points including $0$ and $\infty$.  That is, one has a differential operator of the form $$L=\sum_{j=0}^d t^j P_j(D) = \sum_{i=0}^r q_{r-i}(t) D^i \in \CC[t,D]\mspace{50mu}(\gcd(\{q_{\ell}\})=1),$$ of degree $d$ and order $r$, with singularities only in $\Sigma$, and accompanying $\D:=\D_{\PP^1}$-module $\D/\D L$ on $\PP^1$ with solution sheaf $\mathrm{Hom}_{\D^{\text{an}}}(\D/\D L,\co^{\text{an}}_{\PP^1})$.  Its restriction to $U$ is a \emph{connection} $(\M,\nabla\colon \M\to \M\otimes \Omega^1_U)$ with underlying local system $\MM_{\CC}:=\ker(\nabla^{\text{an}})\subset\M^{\text{an}}$, and solution sheaf $\mathrm{Sol}(\M):=\mathrm{Hom}_{\D^{\text{an}}_U}(\M,\co^{\text{an}}_U)\cong\MM_{\CC}^{\vee}$.  

Write $\mu\in H^0(U,\M)$ for the image of $1\in \D/\D L$, so that $\nabla_L \mu =0$. Local analytic sections $\ve$ of $\MM_{\CC}^{\vee}$ may be paired with $\mu$ to yield \emph{periods} $\langle\ve,\mu\rangle$, which are local sections of $\co_U^{\text{an}}$ satisfying $D\langle\ve,\mu\rangle=\langle\ve,\nabla_D \mu\rangle$ hence $L\langle\ve,\mu\rangle=\langle\ve,\nabla_L\mu\rangle=0$.  On a simply connected subset $\mathcal{S}\subset U^{\text{an}}$, each such period is simply the image of $1\in \D/\D L$ under $\ve$ regarded as an element of $\mathrm{Hom}_{\D_\mathcal{S}}(\D/\D L,\co_{\mathcal{S}})$.

In our setup, the connection is \emph{regular} at $0$ if $\ord_0(q_{\ell})\geq \ord_0(q_0)$ for each $\ell$, at $\sigma\in\Sigma^{\times}:=\Sigma\setminus\{0,\infty\}$ if $\ord_{\sigma}(q_{\ell})\geq \ord_{\sigma}(q_0)-\ell$ for each $\ell$, and at $\infty$ if $\deg(q_{\ell})\leq \deg(q_0)$ for each $\ell$.

\begin{example}\label{E1.1}
Let $\X$ be a smooth projective $(n{+}1)$-fold, $f\colon \X\to \PP^1$ a proper morphism whose restriction $f_U\colon \X_U:=f^{-1}(U)\to U$ is smooth, and consider the exact sequence of complexes $$0\to f_U^* \Omega^1_U \otimes \Omega^{\bullet}_{\X_U/U}[1] \to \Omega_{\X_U}^{\bullet} \to \Omega^{\bullet}_{\X_U/U}\to 0.$$  Applying $\RR^k (f_U)_*$ to its terms yields a long exact sequence in which the (everywhere regular) \emph{Gauss-Manin connection} appears as a connecting homomorphism:  writing $\M:=\RR^n (f_U)_* \Omega_{\X_U/U}^{\bullet}$, we obtain $$\M\overset{\nabla}{\to}\RR^{n+1}(f_U^*\Omega^1_U\otimes\Omega^{\bullet}_{\X_U/U}[1])\cong \Omega^1_U\otimes\RR^n\Omega^{\bullet}_{\X_U/U}=\Omega_U^1\otimes \M.$$ Viewed in the analytic topology, $\nabla$ annihilates $\MM_{\KK}:= R^n (f_U)_*\KK_{\X_U^{\text{an}}}$ for any subring $\KK\subseteq\CC$.  The solution sheaf $\mathrm{Sol}(\M)$ identifies with the local system $\{H_n(X_t,\CC)\}_{t\in U}$.

Without loss of generality, we may assume that $\M$ is irreducible cyclic, so that for some $\mu \in H^0(U,\M)$, $\M$ is generated as an $\co_U$-module by\footnote{In the sequel we make no explicit assumption about $\mu$ generating $\M$ in this strong sense, though in the setting imposed in $\S$\ref{S3}\textit{ff}, it will always generate $\M$ on a smaller Zariski open.} $\mu,\, \nabla_D \mu,\, \nabla^2_D \mu,\,\ldots,\,\nabla^{r-1}_D\mu$. So there exists $L \in \co(U)[D]$, which we may normalize as above, with $\nabla_L \mu =0$.  Local analytic sections $\ve$ of $\MM_{\KK}^{\vee}$ may be paired with $\mu$ to yield $\KK$\emph{-periods} $\langle\ve,\mu\rangle$, refining the ($\CC$-)periods above.
\end{example}

Fix a base point $p\in U(\CC)\cap \RR_{>0}$ near $0$, and a point $\tilde{p}\in\widetilde{U^{\text{an}}}$ above $p$ on the universal cover $\mathpzc{P}\colon\widetilde{U^{\text{an}}}\twoheadrightarrow U^{\text{an}}$.  Also fix paths $\gamma_{\sigma}$ in $U^{\text{an}}$ based at $p$ and winding once counterclockwise about each $\sigma\in \Sigma\setminus\{\infty\}$.  Write $T_{\sigma}$ for the action of monodromy (parallel transport along $\gamma_{\sigma}$) on the stalks $\MM_p$ and $\MM_p^{\vee}$.  In dual bases the matrices of these actions will be transpose-inverse to one another.

\begin{example}\label{E1.2}
Suppose that $\M$ has a regular singularity at $0$, and that $\rk(\MM_p^{T_0})=1$.  Normalizing $\mu$ (and replacing $L$ accordingly), we may assume that the unique invariant period in a neighborhood of $0$ takes the form $A(t)=1+\sum_{k\geq 1}a_k t^k$.  (These two assumptions imply that $P_0(D)=D^r$.) A first motivation for the construction of Bloch-Vlasenko $\Gamma$-functions is:  can we \emph{interpolate} the $\{a_k\}$, i.e. produce an entire function with $F(-m)=a_m$ for all $m\in \ZZ_{>0}$?

If $L=D+t$ then the period is $e^{-t}=\sum_{k\geq 0}\tfrac{(-1)^k}{k!} t^k$, and the sort of function we are after is $$F(s):=\frac{e^{2\pi\ay s}-1}{2\pi\ay}\Gamma(s),\;\;\text{where}\;\; \Gamma(s)=\int_0^{\infty}e^{-t}t^s\frac{dt}{t}.$$ Since $\Gamma(s)\sim\tfrac{(-1)^m}{m!(s+m)}$ for $s\sim -m$, and $\tfrac{e^{2\pi\ay s}-1}{2\pi\ay}\sim s+m$, we get $F(-m)=\tfrac{(-1)^m}{m!}$.  The Bloch-Vlasenko $\Gamma$ in this case would be $(e^{2\pi\ay s}-1)\Gamma(s)$, see Example \ref{E2.1}.
\end{example}

Henceforth (with the exception of Example \ref{E2.1}) we shall \emph{assume that $\M$ has regular singularities}.  Choose a section $\mm\in H^0(U,\M)$, not necessarily the section $\mu$ annihilated by $L$.  For each $\ve\in\MM_{\KK,p}^{\vee}$, by $\langle \ve,\mm\rangle$ we shall mean the holomorphic function on $\widetilde{U^{\text{an}}}$ (or multivalued function on $U^{\text{an}}$) obtained by pairing $\mm$ with the section of $\mathpzc{P}^{-1}(\MM^{\vee}_{\KK})$ extending $\ve$ from $\tilde{p}$.  Let $C_{\bullet}(\widetilde{U^{\text{an}}};\KK)$ be the complex of topological chains on the universal cover; then $$\xi=[\sum_j \gamma_j \otimes \ve_j] \in H_1(U^{\text{an}},\MM_{\KK}^{\vee}):= H_1(C_{\bullet}(\widetilde{U^{\text{an}}};\KK)\otimes_{\KK[\pi_1 (U^{\text{an}},p)]}\MM_{\KK,p}^{\vee})$$ is paired with $\omega =\mm\otimes \tfrac{dt}{t}\in H^1_{\text{dR}}(U,\M)$ by $$\langle\xi,\omega\rangle :=\sum_j \int_{\gamma_j^{-1}}\langle \ve_j,\mm\rangle \frac{dt}{t}.$$  This is called a \emph{period of the connection} $\M$.

\begin{rem}\label{R1.1}
(i) The $H_1$ above also identifies with group homology $H_1(\pi_1(U^{\text{an}},p),\MM_{\KK,p}^{\vee})$, computed by the complex $\mathcal{C}_2\to\mathcal{C}_1\to\mathcal{C}_0$, where $\mathcal{C}_0:=\MM_{\KK,p}^{\vee}$ and (for $n=1,2$) $$\mathcal{C}_n :=\{\text{free abelian group on symbols }[g_1,\ldots,g_n]\}\otimes \MM_{\KK,p}^{\vee}.$$ The differential is given by $\partial([\gamma_1,\gamma_2]\otimes \ve)=[\gamma_2]\otimes \gamma_1^{-1}\ve-[\gamma_1\gamma_2]\otimes \ve+[\gamma_1]\otimes\ve$ and $\partial([\gamma]\otimes\ve)=\gamma^{-1}\ve-\ve$, which reflects the multivaluedness of the sections of $\MM^{\vee}_{\KK}$.

(ii) The pairing is well-defined: if $\xi\in\partial\mathcal{C}_2$ holds or $\omega$ is a dR-coboundary, then $\langle\xi,\omega\rangle=0$.  In the first case, this follows from $$\int_{(\gamma_1\gamma_2)^{-1}}\langle \ve,\omega\rangle\frac{dt}{t}=\int_{\gamma_2^{-1}\gamma_1^{-1}}\langle \ve,\omega\rangle = \int_{\gamma_1^{-1}}\langle \ve,\omega\rangle\frac{dt}{t}+\int_{\gamma_2^{-1}}\langle \gamma_1^{-1}\ve,\omega\rangle\frac{dt}{t},$$ which holds because $\gamma_1^{-1}$ has acted on $\ve$ before we start along $\gamma_2^{-1}$.  For the second, if $\omega=\nabla\eta=\nabla_D\eta\otimes\tfrac{dt}{t}$ then
\begin{align*}
\langle\xi,\omega\rangle &= \sum_j\int_{\gamma_j^{-1}}\langle \ve_j,\nabla_D\eta\rangle \frac{dt}{t} = \sum_j\int_{\gamma_j^{-1}}D\langle\ve_j,\eta\rangle \frac{dt}{t}\\
&=\langle \sum_j (\gamma_j^{-1}\ve_j - \ve_j),\eta\rangle = \langle 0,\eta\rangle=0
\end{align*}
by the Fundamental Theorem of Calculus.
\end{rem}

\section{Gamma functions and interpolation}\label{S2}

Consider the rank-1 connection on $\co_U$ with $\nabla_D 1:=s$, so that the differential operator is $D-s$ and the period is $t^s$.  By abuse of notation we write this connection as ``$t^s$'', and set $\M(s):=\M\otimes t^s$.  The action of $\pi_1(U^{\text{an}},p)$ on its stalk $\MM(s)_{\KK,p}^{\vee}=\MM_{\KK,p}^{\vee}\otimes_{\KK} \KK[e^{\pm2\pi\ay s}]$ is the tensor product of the monodromy representation for $\MM^{\vee}_{\KK}$ with the monodromy of $t^s=e^{s\log t}$ on $\CC^*$.  (We take $1\in \KK[e^{\pm 2\pi\ay s}]$ to correspond to the branch with $\log(p)\in \RR$.)  Our interest lies in certain periods of this ``Mellin-transformed'' connection:

\begin{defn}\label{D2.1}
Given $\mm\in \M(U)$ and $\xi = [\sum_j \gamma_j \otimes \ve_j \otimes e^{2\pi\ay n_j s}] \in H_1(U^{\text{an}},\MM(s)_{\KK}^{\vee})$, with $n_j\in\ZZ$, the associated Bloch-Vlasenko \emph{Gamma function} is $$\Gamma_{\xi,\mm}(s):=\sum_j e^{2\pi\ay n_j s}\int_{\gamma_j^{-1}}\langle \ve_j,\mm\rangle t^s \frac{dt}{t}.$$ It is called \emph{motivic} if $\M$ arises as in Example \ref{E1.1}.
\end{defn}

\begin{rem}\label{R2.1}
(i) This function is entire: $\cup|\gamma_j|$ avoids singularities of the integrand, which is thus uniformly bounded for $s$ in any compact set.

(ii) Given $\mm$, $\Gamma_{\xi,\mm}$ depends only on $\xi$ (and not its representative) by Remark \ref{R1.1}(ii) applied to $\M(s)$, with $\omega=\mm\otimes 1\otimes \tfrac{dt}{t}$.  Hence the set of all Gamma functions for $(\M,\mm)$ is an \emph{image} of $H_1(U,\MM(s)_{\KK}^{\vee})$, and is finitely generated as a $\KK[e^{\pm2\pi\ay s}]$-module.  
\end{rem}

Recall that $\mu$ is the section of $\M$ annihilated by $L=\sum_{k=0}^d t^k P_k(D)$.

\begin{thm}\label{T2.1}
The Gamma functions for $(\M,\mu)$ satisfy the difference equation $$\sum_{k=0}^d P_k(-s-k)\Gamma_{\xi,\mu}(s+k)=0.$$
\end{thm}

\begin{proof}
Applying the FTC to $0=\partial\xi=\sum_j e^{2\pi\ay n_j s}(\gamma_j^{-1}-1)(\ve_j \otimes 1)$ yields $0=\sum_j e^{2\pi\ay n_j s}\int_{\gamma_j^{-1}} D(\langle\ve_j,\mu\rangle t^s)\tfrac{dt}{t} = \Gamma_{\xi,\nabla_D\mu}(s)+s\Gamma_{\xi,\mu}(s).$ Moreover, $\Gamma_{\xi,t\mu}(s)=\Gamma_{\xi,\mu}(s+1)$ is evident from the definition. So $\sum_j t^j P_j(\nabla_D)\mu=0$ gives $0=\Gamma_{\xi,0}(s)=\sum_j \Gamma_{\xi,t^j P_j(\nabla_D)\mu}(s)=\sum_j \Gamma_{\xi,P_j(\nabla_D)\mu}(s+j)=\sum_j P_j(-s-j)\Gamma_{\xi,\mu}(s+j)$.
\end{proof}

\begin{rem}[Recurrence relations] \label{R2.2}
In the setting of Example \ref{E1.2}, we have 
\begin{align*}
0&= L A(t) = \sum_{k=0}^d t^k P_k(D)\sum_{m\geq 0}a_m t^m= \sum_{k=0}^d \sum_{m\geq 0} P_k(m) a_m t^{m+k} \\ &= \sum_{m\geq 0}\left( \sum_{k=0}^d P_k (m-k) a_{m-k}\right) t^m
\end{align*}
hence $\sum_{k=0}^d P_k(m-k)a_{m-k}=0$ for all $m$, which determines $a_m$ from $\{a_{m-k}\}_{k=1}^{\min\{m,d\}}$.  Setting $s=-m$ in Theorem \ref{T2.1}, we have $$\sum_{k=0}^d P_k(m-k)\Gamma_{\xi,\mu}(-m+k)=0.$$ So if we assume $\Gamma_{\xi,\mu}(0)=2\pi\ay$, and $\Gamma_{\xi,\mu}(\ell)=0$ for $\ell\in\ZZ_{>0}$, then $\Gamma_{\xi,\mu}(-m)=2\pi\ay a_m$. As we shall see, in the confluence of the settings of Examples \ref{E1.1} and \ref{E1.2}, these formulas will turn out to be true up to a nonzero rational factor.  Therefore, the Bloch-Vlasenko $\Gamma$-function \emph{interpolates} the power-series coefficients $\{a_m\}$.
\end{rem}

To conclude with the ``simplest example'', we have to break the rule about regular singularities.

\begin{example}\label{E2.1}
Let $\M$ be the connection on $\co_{\mathbb{G}_m}$ with $\nabla_D 1=-t$.  The differential operator is $D+t$, its period $e^{-t}$ ($=\langle \ve,1\rangle$ for a section $\ve$ of $\MM^{\vee}$).  Consider the path $\gamma$ which runs from $\infty$ to $\e>0$ along $\RR_{>0}$, once counterclockwise around $0$, then back to $\infty$ along $\RR_{>0}$.  Due to the subpolynomial decay of $e^{-t}$ at $\infty$,  $\xi = \gamma\otimes \ve\otimes 1$ is a ``rapid decay cycle'' in $H_1^{\text{RD}}(\CC^*,\MM(s)^{\vee})$ (see \cite{BE}), so that $$\Gamma_{\xi,1}(s)=\int_{\gamma}\langle \ve,1\rangle t^s \frac{dt}{t}=(e^{2\pi\ay s}-1)\int_0^{\infty} e^{-t} t^{s-1}dt = (e^{2\pi\ay s}-1)\Gamma(s)$$ as advertised.  But this is ``ur-Gamma'' is not a \emph{motivic} Gamma!
\end{example}

\section{Conifold monodromy}\label{S3}

\textbf{Henceforth we work in the following setting}, which is motivated and typified by the simplest $\D$-modules arising from Landau-Ginzburg models:

\begin{itemize}[leftmargin=1cm]
\item $(\M,\nabla)$ is \emph{motivic}, which is to say that it underlies a sub-$\QQ$-PVHS of an $\RR^n (f_U)_* \Omega_{\X_U/U}^{\bullet}$ (defined as in Example \ref{E1.1}).  This implies:
  \begin{itemize}
  \item $\M$ has regular singularities;
  \item fiberwise $\QQ$-Betti cohomology provides a $\QQ$-local system $\MM_{\QQ}$ underlying $\MM$, whose monodromies $T_{\sigma}=T_{\sigma}^{\text{ss}}e^{N_{\sigma}}$ are thus defined over $\QQ$;
  \item fiberwise integration yields a polarization $Q(\cdot,\cdot)\colon \M\times\M\to \co$ sending $\MM_{\QQ}\times\MM_{\QQ}\to\QQ$;\footnote{That is, $Q$ is a morphism of VHS \emph{of weight} $-2n$.  The induced isomorphism $Q(\cdot)\colon \M\to\M^{\vee}$ defined by $Q(a,b)=\langle Q(a),b\rangle$ sends $\MM_{\QQ}\to\MM_{\QQ}^{\vee}$; and the polarization on $\M^{\vee}$ defined by $Q(a,b):=\langle a,Q^{-1}(b)\rangle$ restricts to the intersection form on $\QQ$-Betti homology $\MM_{\QQ}^{\vee}$.  (The ``missing'' $(2\pi\ay)^n$ twist will eventually show up.)} and
  \item $\M$ has a varying Hodge flag $\fb$, with $\nabla \fb \subset \F^{\bullet-1}\otimes \Omega^1_U$, satisfying the Hodge-Riemann relations.
  \end{itemize}
We will use $\M$ also to denote this PVHS in what follows.
\item $\M$ is \emph{principal}: the $\gr_{\F}^p \M$ are all of rank $1$ for $p=0,1,\ldots,n$, so that the rank of $\M$ and order of $L$ are related by $r=n+1$.
\item $\M$ has \emph{maximal unipotent monodromy} at $t=0$:  $\rk(\MM^{T_0})=1$.  Accordingly, fixing $\ve\in(\MM_{\QQ,p}^{\vee})^{T_0}$ once and for all, there exists a basis $\ve_0,\ve_1,\ldots,\ve_n$ of $\MM_{\QQ,p}^{\vee}$ with $N_0\ve_i=\ve_{i-1}$.  Though this basis is not unique, $\mathsf{Q}_0:=Q(\ve_0,\ve_n)\in \QQ^{\times}$ is independent of the choice (which in any case we will specify below).
\item There is a ``minimal'' $c\in\Sigma^{\times}$, with $|c|<|\sigma|$ for all other $\sigma\in\Sigma^{\times}$; and $\M$ has \emph{conifold monodromy} at $t=c$:  $\rk(T_c - I)=1$.  That is, there exists $\delta\in \MM_{\QQ,p}^{\vee}$ such that:
  \begin{itemize}
  \item the linear span $\langle \delta\rangle = \mathrm{im}(T_c - I)\MM_{\QQ,p}^{\vee}$;
  \item for $n$ odd, $T_c$ is a symplectic transvection, sending $\delta\mapsto\delta$ and some $\beta\mapsto\beta+\delta$;
  \item for $n$ even, $T_c\colon \delta\mapsto -\delta$ is an orthogonal reflection; and
  \item $\ve\in\MM_{\QQ,p}^{\vee}$ is invariant under $T_c$ if and only if $Q(\ve,\delta)=0$.
  \end{itemize}
\item Finally, assume that $T_c \ve_0 \neq \ve_0$.  We may then rescale $\delta$ so that $(T_c-I)\ve_0=\delta$, and set $\mathsf{Q}_c:=Q(\ve_0,\delta)\neq 0$.
\end{itemize}

\noindent Writing $T_{\sigma}=T_{\sigma}^{\text{un}} T_{\sigma}^{\text{ss}}$ for the Jordan decomposition and $N_{\sigma}:=\log(T_{\sigma}^{\text{un}})$ for the monodromy logarithms, the assumptions just made imply $T_0=e^{N_0}$ and $N_0 \ve_0=0$, as well as:
\begin{lem}\label{L3.1}
$\delta$ generates $\MM_p^{\vee}$ under $N_0$.
\end{lem}
\begin{proof}
First note that if $i+j<n$, then $n-i>j$ $\implies$ $N^{n-i}\ve_j =0$ $\implies$ $Q(\ve_i,\ve_j)=Q(N^{n-i},\ve_n,\ve_j)=(-1)^{n-1}Q(\ve_n,N^{n-i}\ve_j)=0$.  (In particular, $Q(\ve_0,\ve_k)=0$ for $k<n$; and since $Q$ is nondegenerate, we must then have $Q(\ve_0,\ve_n)\neq 0$ as mentioned in the third bullet above.)

Now suppose that $\delta=\sum_{i\leq k} \mathtt{c}_i\ve_i$, with $k<n$. Then for any $\ve\in \MM_{\QQ,p}^{\vee}$, $Q( \ve,(T_c^{-1}-I)\ve_j)=Q((T_c-I)\ve,\ve_j) = Q(\mathtt{c}_{\ve}\delta,\ve_j)=\sum_{i\leq k}\mathtt{c}_i \mathtt{c}_{\ve} Q(\ve_i,\ve_j)$ is $0$ for all $j<n-k$. Hence $\ve_0,\ldots,\ve_{n-k-1}$ are $T_c$-invariant, which in the case of $\ve_0$ contradicts the last bullet above.
\end{proof}

\noindent Before proceeding, we make some final calibrations to the $\QQ$-Betti homology classes as follows:
\begin{lem}\label{L3.1b}
Given $\ve_0$, there exists a unique choice of $\ve_1,\ldots,\ve_n$ satisfying $N_0\ve_j=\ve_{j-1}$ and $(T_c-I)\ve_j=0$ for $j>0$.
\end{lem}

\begin{proof}
Given initial choices $\vc_1,\ldots,\vc_n$ (and $\vc_0=\ve_0$) satisfying $N_0\vc_j=\vc_{j-1}$, write $(T_c-I)\vc_k =:\mathsf{d}_k \delta$ (with $\mathsf{d}_0=1$), and inductively define $\ve_k:=\vc_k -\sum_{j=1}^k \mathsf{d}_j\ve_{k-j}$ for $k=1,\ldots,n$.  One easily checks the desired properties (by induction).

Suppose $\ve_1 ',\ldots, \ve_k '$ also satisfy the two properties in the statement of the Lemma.  Inductively assuming that $\ve_i ' = \ve_i$ for $i<k$, we have $N_0 (\ve_k ' -\ve_k)=\ve_{k-1}'-\ve_{k-1}=0$ hence $\ve_k ' =\ve_k + a\ve_0$; whence $0=(T_c-I)\ve_k ' =(T_c-I)\ve_k + a(T_c-I)\ve_0 =a\delta$ $\implies $ $a=0$.
\end{proof}

\begin{rem}\label{R3.1}
In the event that the geometry $\X_U\to U$ underlying $\M$ extends over $c$ to a degeneration with smooth total space and nodal singular fiber $X_c$, we will say that $\M$ has \emph{strong} conifold monodromy at $c$.  In this case, there is a conifold vanishing sphere $\delta_0$ with $Q(\delta_0,\delta_0)=(-1)^{\binom{n+1}{2}}(1+(-1)^n)$, which controls the monodromy via the Picard-Lefschetz formula $T_c \ve = \ve - (-1)^{\binom{n}{2}} Q(\ve,\delta_0)\delta_0$.  (We then have $\delta=M\delta_0$ for some $M\in\QQ^{\times}$, and $\mathsf{Q}_c = -(-1)^{\binom{n}{2}}M^2$.) We shall only assume this where indicated, since there are times when one merely has a differential operator in hand.
\end{rem}

Turning to the de Rham structure, let $\M_e$ be the canonical extension of $\M$ to $\PP^1$, whose logarithmic connection $\nabla\colon\M_e \to \M_e\otimes \Omega^1_{\PP^1}\langle \log \Sigma\rangle$ has residues $\mathrm{Res}_{\sigma}(\nabla) = -\tfrac{N_{\sigma}}{2\pi\ay}-{\mathrm{Log}}(T_{\sigma}^{\text{ss}})$ (with $\mathrm{Log}$ the branch of $\tfrac{\log}{2\pi\ay}$ with real part in $[0,1)$).  The extended Hodge sub-bundles $\fb_e$ satisfy $\nabla(\F_e)\subseteq \F^{\bullet-1}_e \otimes \Omega^1_{\PP^1}\langle\log\Sigma\rangle$.  In particular, the line bundle $\F_e^n$ is positive, and so has nonzero holomorphic sections; we \emph{take $\mu\in H^0(\PP^1,\F_e^n)$ to be the unique such section with zeroes only at $\infty$ and normalized so that $\langle\ve_0,\mu\rangle=A(t)=\sum_{m\geq 0}a_m t^m$ has $a_0=1$}. The assumption that $T_c \ve_0\neq \ve_0$ implies that $A(t)$ has monodromy at $c$, and so $\limsup_{m\to \infty}a_m^{1/m}=c^{-1}.$  

Henceforth $L=\sum_{j=0}^d t^j P_j(D) = \sum_{i=0}^r q_{r-i}(t) D^i$ \emph{shall denote the \textup{(}Picard-Fuchs\textup{)} differential operator associated to this} $\mu$, written so that the $\{q_i\}_{i=0}^r$ have no common factor.  That is, $L$ annihilates $\mu$ and all of its periods.  (From this point onward we shall drop $\nabla$ when convenient, writing $D\mu$ etc.)  We shall be interested in the particular $\QQ$-periods
\begin{align*} \e_k(t):&=\langle \ve_k,\mu\rangle \mspace{50mu} (k=0,1,\ldots,n)\\ \psi(t):&=\langle \delta,\mu\rangle , \end{align*} where of course $\e_0(t)=A(t)$.  Recalling that $g(t)\sim h(t)$ at $t=\sigma$ means $\lim_{t\to \sigma}\tfrac{g(t)}{h(t)}=1$, here is what we can say about their asymptotic behavior:

\begin{lem}\label{L3.2}
\textup{(i)} At $t=0$, $\e_k(t) \sim \frac{\log^k(t)}{k! (2\pi\ay)^k}$.

\textup{(ii)} Write $\mathrm{E}_n(z):=z^{\frac{n-1}{2}}$ for $n$ even and $z^{\frac{n-1}{2}}\log(z)$ for $n$ odd.  If $\M$ has strong conifold monodromy, then about $t=c$ we have 
\begin{align*}
\e_0(t) &= C_0(1+\mathcal{O}(t-c))\mathrm{E}_n(t-c)+\text{analytic function} \\ \text{and}\;\;\;\psi(t) &\sim C (t-c)^{\frac{n-1}{2}}
\end{align*} 
for some constants $C_0,C\in \CC^{\times}$, and $\ord_{t=c}(q_0)=1$.
\end{lem}
\begin{proof}
Applying repeatedly that $(2\pi\ay)D\langle \ve_k,\mu\rangle =(2\pi\ay)\langle \ve_k,\nabla_D \mu\rangle$ is asymptotic to $(2\pi\ay)\langle\ve_k,(\mathrm{Res}_0\nabla)\mu\rangle = -\langle \ve_k,N_0\mu\rangle=\langle N_0\ve_k,\mu\rangle =\langle \ve_{k-1},\mu\rangle$ yields (i).  For (ii), the period exponent of a node $x_0^2+\cdots +x_n^2$ is $\tfrac{n+1}{2}$ (see \cite[(4.6-7) and Prop. 4.1]{KL2}), and by the assumptions above $\ve_0$ maps onto the (rank one) vanishing cohomology.   Since $X_c$ is still $K$-trivial, and $\mu_c$ nonvanishing as a section of $\F_{e,c}^n\cong H^0(K_{X_c})$, the period $\e_0=\int_{\ve_0}\mu$ realizes this exponent in [op. cit., (4.6)], yielding the claim about $\e_0$.  For $\psi$, use $(T_c-I)\e_0=\psi$.

Choose a local coordinate $w\sim t-c$ so $\e_0=\mathrm{E}_n(w)+$analytic terms, and write $\d=\tfrac{d}{dw}$.  Then $\M_{e,c}$ is generated by
\begin{align*}
&\mu,\;\d\mu,\;\ldots,\d^{\frac{n-1}{2}}\mu,\;w\d^{\frac{n+1}{2}}\mu,\;\d w\d^{\frac{n+1}{2}}\mu,\ldots,\;\d^{\frac{n-1}{2}}w\d^{\frac{n+1}{2}}\mu\;\;\;\;\;\text{resp.}\\
&\mu,\;\d\mu,\;\ldots,\;\d^{\frac{n}{2}}\mu,\; w^{\frac{1}{2}}\d w^{\frac{1}{2}}\d^{\frac{n}{2}}\mu,\;\d w^{\frac{1}{2}}\d w^{\frac{1}{2}}\d^{\frac{n}{2}}\mu,\ldots,\;\d^{\frac{n}{2}-1}w^{\frac{1}{2}}\d w^{\frac{1}{2}}\d^{\frac{n}{2}}\mu,
\end{align*}
and $\d^{\frac{n+1}{2}}w\d^{\frac{n+1}{2}}\mu$ resp. $\d^{\frac{n}{2}}w^{\frac{1}{2}}\d w^{\frac{1}{2}}\d^{\frac{n}{2}}\mu$ belong to $\M_{e,c}$.  From this one deduces that $w\d^{n+1}\mu$ (and not $\d^{n+1}\mu$) is a $\CC[w]$-linear combination of $\mu,\d\mu,\ldots,\d^n \mu$.
\end{proof}

Here is a basic geometric example invoked repeatedly in \S\S\ref{S7}-\ref{S8}.

\begin{example}\label{E3.1}
Let $\vf\in \CC[x_1^{\pm1},\ldots,x_{n+1}^{\pm 1}]$ be a Laurent polynomial whose Newton polytope $\NP$ is reflexive, i.e. has integral polar polytope.  (In particular, it has a unique integral interior point given by $\uo$.)  We shall call $\vf$ itself \emph{reflexive} if in addition there exists a smooth blowup $\beta\colon\cx\twoheadrightarrow\PP_{\NP}$ on which $\tfrac{1}{\vf}$ extends to a proper morphism $f\colon\cx\to\PP^1$, $X_0=f^{-1}(0)\subset \cx$ is a normal-crossing divisor, and ($\beta^*$ of) $\mathrm{dlog}(\ux):=\tfrac{dx_1}{x_1}\wedge\cdots\wedge\tfrac{dx_{n+1}}{x_{n+1}}$ extends to a nowhere-vanishing section of $\Omega_{\cx}^{n+1}(\log\,X_0)$. Writing $\Sigma$ for the discriminant locus of $f$, $X_t:=f^{-1}(t)$ is a smooth CY $n$-fold for each $t\in \PP^1\setminus \Sigma\,(=:U)$, and is given by a crepant resolution of $\overline{\{1-t\vf(\ux)=0\}}\subset \PP_{\NP}$.  We call $(\cx,f)$ the \emph{compact Landau-Ginzburg model} associated to $\vf$.

Put $\cx_U:=f^{-1}(U)$, so that $(\RR^n {f_U}_*\Omega_{\cx_U/U},\nabla)$ underlies a $\QQ$-VHS $\H^n_f$, and write $\M\subseteq \H_f^n$ for the minimal sub-$\QQ$-VHS containing the line bundle $\F^n:=\F^n\H^n_f$.  If $\M$ satisfies the assumptions of the beginning of this section, we will call $\vf$ \emph{good}. Taking a section $\mu$ of $\F_e^n$ as above with corresponding Picard-Fuchs operator $L$, it is enough to have $L$ of order $n+1$ with unique exponent $0$ at $t=0$\footnote{This condition forces the underlying local system to be rational, since it implies $N_0^n\neq 0$, and a Galois-conjugate system inside $\HH$ could not also have this property (since $h^{n,0}=1$).} and a single integer exponent of multiplicity two or half-integer exponent of multiplicity one (for $n$ odd resp. even) at $t=c$.

In fact, we can identify the section $\mu$ explicitly.  Denoting by $\omega_f:=\omega_{\cx}\otimes f^* \omega_{\PP^1}^{-1}$ the relative dualizing sheaf, by a result of Koll\`ar \cite[Thm 2.6]{Ko} we have $\F_e^n \cong f_*\omega_f$.  Clearly $\tfrac{\mathrm{dlog}(\ux)}{f^*(dt/t)}$ is a section of $\omega_f\cong \omega_{\cx}(\log\, X_0)\otimes f^*\omega_{\PP^1}(\log\,0)^{-1}$ 
vanishing to first order on $X_{\infty}$ and nowhere else. Hence $\mu:=\left[\tfrac{1}{(2\pi\ay)^n}\tfrac{\mathrm{dlog}(\ux)}{df/f}\right] \in H^0(\PP^1,\F^n_e)$ is a section with a simple zero at $\infty$, demonstrating that $\F_e^n\cong \co(1)$.  Moreover, for each $t\in U$ we have $\tfrac{1}{(2\pi\ay)^n}\tfrac{\mathrm{dlog}(\ux)}{1-t\vf}=\tfrac{\mu\wedge df/f}{1-t\vf}=\tfrac{\mu\wedge df}{f-t}=\mu\wedge \mathrm{dlog}(f-t)$ $\implies$
 $$\mu_t=\tfrac{1}{(2\pi\ay)^n}\mathrm{Res}_{X_t}\left(\tfrac{\mathrm{dlog}(\ux)}{1-t\vf(\ux)}\right)\in\Omega^n (X_t).$$ From this one easily shows (e.g. see \cite[(4.1)]{DK}) that $a_m$ is the constant term in $\vf^m$; in particular, $a_0=1$ as desired.

Finally, we can broaden this construction by allowing Laurent polynomials which define families with an automorphism over $t\mapsto e^{\frac{2\pi\ay}{w}}t$ for some $w\in \mathbb{N}$, and which fail to be good only insofar as there are $w$ conifold points of minimal modulus in $\Sigma$.  Replacing $\X$ with its quotient by this automorphism and $t$ by $t^w$, and assuming the new $T_0$ remains unipotent, $\mu$ still produces the desired section.  In the sequel, all constructions and results stated for good reflexive polynomials $\vf$ are also valid in this setting.
\end{example}

\section{Frobenius periods}\label{S4}

Since $\MM$ has maximal unipotent monodromy at $t=0$ and $A(0)\neq 0$, it follows that $L$ has the unique local exponent $0$ there.  The indicial equation $P_0(\mathtt{T})=0$ thus has unique root $\mathtt{T}=0$, and so $P_0(D)=D^r$.

\begin{defn}\label{D4.1}
A \emph{Frobenius deformation} for $L$ at $0$ is a formal series $$\Phi(s,t)=\sum_{m\geq 0}\phi_m(t) s^m,$$ with each $\phi_m$ analytic on a neighborhood of $p$ (and by continuation, on $\widetilde{U^{\text{an}}}$), such that $L\Phi=s^r t^s$ and $T_0\Phi=e^{2\pi\ay s}\Phi$.  We shall call $\phi_0,\ldots,\phi_n$ the \emph{Frobenius periods}, since they satisfy $L(\cdot)=0$.
\end{defn}

\noindent In our setting (as bulleted in $\S$\ref{S3}), $\Phi$ is unique \cite{BV}.

\begin{example}\label{E4.1}
If $L$ has order $3$ ($n=2$), then $$L(\sum_{m\geq 0}\phi_m s^m)=s^3 e^{s\log t} = \sum_{m\geq 3} \tfrac{\log^{m-3}(t)}{(m-3)!}s^m $$ implies $L\phi_0, L\phi_1, L\phi_2 = 0$ (morally, 3 $\CC$-periods of a family of Picard rank $19$ $K3$ surfaces) while $L\phi_3 =1$, $L\phi_4= \log(t)$, $L\phi_5 = \tfrac{\log^2 t}{2!}$, etc.
\end{example}

The monodromy condition $T_0\Phi = e^{2\pi\ay s}\Phi$ forces $t^{-s}\Phi$ to be $T_0$-invariant (after expanding $t^{-s}=e^{-s\log(t)}$ and rearranging in powers of $s$). Since the $\phi_m$ have at worst log poles, the coefficients $\{\phi^{\text{an}}_m\}$ of powers of $s$ in $t^{-s}\Phi$ are thereby analytic in a disk about $t=0$.  Writing $\phi_m(t)=\sum_{k\geq 0}a^{(m)}_k t^k$, we have
\begin{align*}
\Phi(s,t)&=\sum_{\mathtt{m}\geq 0} \phi^{\text{an}}_{\mathtt{m}}(t) t^s s^{\mathtt{m}}=\sum_{m\geq 0}s^m \sum_{\ell\geq 0}\tfrac{\log^{\ell}t}{\ell!}\phi^{\text{an}}_{m-\ell}(t) \\
&= \sum_{j,k,\ell\geq 0} s^{j+\ell}  \tfrac{\log^{\ell}t}{\ell!}a^{(j)}_k t^k = \sum_{k\geq 0} t^k e^{s\log(t)}  \sum_{j\geq 0}a^{(j)}_k s^j =: \sum_{k\geq 0} t^{k+s} A_k(s)
\end{align*}
in which the first line yields $\phi_m(t)=\sum_{\ell=0}^m \tfrac{\log^{\ell}t}{\ell!}\phi^{\text{an}}_{m-\ell}(t)$.  Furthermore, taking $t=0$ in
\begin{align*}
s^r&= t^{-s} L\Phi = t^{-s} (D^r+t(\cdots))\textstyle{\left(\sum_{j\geq 0} a_0^{(m)} s^m + t(\cdots)\right)t^s} \\&= \sum_{j\geq 0} a_0^{(j)} s^{j+r}+t(\cdots)
\end{align*}
gives $1=\sum_{j\geq 0}a_0^{(j)} s^j$, so that $a_0^{(j)}=\bm{\delta}_{0j}$.  Immediate consequences are that $A_0(s)=\sum a_0^{(j)}s^j =1$, and (from $a_0^{(0)}=1$ and uniqueness of the holomorphic period) that $\phi_0(t)=\phi_0^{\text{an}}(t)=\e_0(t)=A(t)$, so that $a_k^{(0)}=a_k$.

\begin{rem}\label{R4.1}
\emph{A priori} the $\{A_k(s)\}_{k>0}$ and $\Phi(s,t)$ are formal in $s$.  However, $L\sum_{k\geq 0}t^{k+s} A_k(s)=s^r t^s$ implies the recurrence $$A_m(s)=-(m+s)^{-r}\,\textstyle{\sum_{j=1}^d} A_{m-j}(s) P_j (m-j+s),$$ where $\deg(P_j)\leq r$ for each $j$ and $A_k(s):=0$ for $k<0$.  This exhibits $A_m(s)$ as a rational function with poles (of order $\leq r$) in $\ZZ\cap [-m,-1]$.  Moreover, the asymptotics of $A_m(s)$ as $m\to \infty$ are governed by the degree $r$ terms of the $\{P_j\}$; these are the coefficients of $q_0(t)$, whose smallest root is nothing but $c$.  One deduces that: for $s$ in any compact subset of $\CC\setminus \ZZ_{<0}$ and $t$ in any disk about $0$ of radius less than $|c|$, the series $\Phi^{\text{an}}:=t^{-s}\Phi =\sum_k A_k(s) t^k$ converges uniformly to an analytic function; and $\Phi^{\text{an}},\Phi$ continue to analytic functions on $\widetilde{U^{\text{an}}}\times (\CC\setminus\ZZ_{<0})$.

We note here for reference the consequences that $\Phi(0,t)=A(t)$ and $\Phi^{\text{an}}(s,0)=1$; from the latter, one has for example that $\Phi(\ell,t)$ is an analytic function vanishing at $t=0$ for each $\ell\in \ZZ_{>0}$.
\end{rem}

\begin{rem}\label{R4.2}
In view of the equality of the $0^{\text{th}}$ Frobenius and $\QQ$-Betti periods $\phi_0(t)=A(t)=\e_0(t)$, it is natural to ask whether the remaining Frobenius periods are $\QQ$-periods.  It turns out that if this were the case, then the limiting mixed Hodge structure (LMHS) of $\M$ at $t=0$ would be $\QQ$-split, \emph{without} even renormalizing $t$!  This is almost never true.

To see the relationship, recall that the LMHS is given by the limiting Hodge flag $\lim_{t\to 0}e^{\el(t)N_0}\fb_t$ written with respect to the $\QQ$-basis $\ve_0^{\vee},\ldots,\ve_n^{\vee}$, together with the weight monodromy filtration $W(N_0)_{2j}=\langle \ve_n^{\vee},\ldots,\ve_{n-j}^{\vee}\rangle$.  But for computing the periods of the LMHS it is better to apply $e^{-\el(t)N_0}$ to the $\QQ$-basis and compare with $\fb_{e,0}$ in the limit.  More precisely, for us the \emph{period matrix of the LMHS}, written $\Omega_{\lim}$, is the change-of-basis matrix between\footnote{Here $(e^{-\el(t)N_0} \ve_j^{\vee})|_{t=0}$ belongs to $W(N_0)_{2(n-j)}\M_{\lim,\QQ}$, and $\nabla_D^j \mu$ to $\M_{\lim}^{n-j,n-j}$; the $(2\pi\ay)^{-j}$ rescaling makes them project to the same element of $\gr^{W(N_0)}_{2(n-j)}\M_{\lim}$.} $\{(2\pi\ay)^{-j} e^{-\el(t)N_0}\ve_j^{\vee} \}_{j=0}^n$ and $\mu,\nabla_D\mu,\ldots,\nabla_D^n \mu$ at $t=0$.  Its $0^{\text{th}}$ column is $$\lim_{t\to 0}\langle (2\pi\ay)^j e^{-\el(t)N_0}\ve_j,\mu\rangle=(2\pi\ay)^j \lim_{t\to 0}\e_j^{\text{an}}(t)=(2\pi\ay)^j \e^{\text{an}}_j(0),$$ where $\e_j^{\text{an}}(t)$ is the ``analytic part'' obtained from $\e_j(t)$ by formally setting $\log(t)$ to zero.  Since $N_0 \ve_j^{\vee}=-\ve_{j+1}^{\vee}$ and $\mathrm{Res}_0(\nabla)=\tfrac{-N_0}{2\pi\ay}$, each column is obtained from the previous one by shifting the entries down, yielding a lower-triangular matrix with ones on the diagonal.

If the $\{\phi_j\}$ were $\QQ$-linear combinations of the $\{\e_j\}$, the $\{\e_j^{\text{an}}\}$ would be $\QQ$-linear combinations of the $\{\phi_j^{\text{an}}\}$. Since $\phi_j^{\text{an}}(0)=\bm{\delta}_{0j}$, all $\e_j^{\text{an}}(0)$ would be rational, and the $(j,j-\ell)^{\text{th}}$ entries of the matrix would belong to $\QQ(\ell)$, making the LMHS $\QQ$-split.
\end{rem}

\section{The kappa series}\label{S5}

We now turn to the analytic continuation of the Frobenius deformation around the conifold point. 
If $\cl =\sum_{i=0}^{\r}\q_{\r-i}(t)D^i$ is a differential operator underlying an algebraic connection, then its adjoint $$\cl^{\dagger}:=(-1)^{\r} \,\textstyle{\sum_{i=0}^{\r}} (-D)^i \q_{\r-i}(t)$$ underlies the dual connection \cite[Lemma 34]{BV}.  (In a slight abuse of notation, we shall write $\mathrm{Sol}_p(\cl)$ for the stalk $\mathrm{Sol}_p(\D/\D\cl)$ below.)  Note that $(\cl^{\dagger})^{\dagger}=\cl$ and $(D\cl)^{\dagger}=\cl^{\dagger} D$.
 
Now remember that $\psi =\langle\delta,\mu\rangle$ denotes the period over the conifold vanishing cycle.  If $\cl$ satisfies $(T_c -I)\mathrm{Sol}_p(\cl)=\CC\psi$, then also $(T_c -I)\mathrm{Sol}_p(\clv)$ has rank one; and since $\mathrm{Sol}_p(\clv D)=\int \mathrm{Sol}_p(\clv)\tfrac{dt}{t}$, $(T_c - I)\mathrm{Sol}_p (\clv D)=\oint_{\gamma_c}\mathrm{Sol}_p(\clv)\tfrac{dt}{t}$ has rank one too.  (That is, all but one function in a basis of $\mathrm{Sol}_p(\clv)$ is analytic at $c$.)  Therefore $(T_c-I)\mathrm{Sol}_p(D\cl)=\CC\psi$.  Applying this argument to get from $\cl=D^{k-1}L$ to $D^k L$, we find that $(T_c -I)\mathrm{Sol}_p(D^k L)=\CC\psi$ for all $k\in\ZZ_{\geq 0}$. But the coefficients $\phi_m$ of $\Phi=\sum_{m\geq 0}\phi_m s^m$ satisfy $D^{\ell}L\phi_m =0$ for $m<\ell+r$, hence $(T_c-I)\phi_m=\kappa_m\psi$ for some $\kappa_m\in \CC$.  (In particular, by the normalization in $\S$\ref{S3}, we have $\kappa_0 =1$.)  So the following makes sense:

\begin{defn}\label{D5.1}
The \emph{kappa series} $\kappa(s)=\sum_{j\geq 0}\kappa_j s^j$ of $L$ is the analytic function on $\CC\setminus \ZZ_{<0}$ given by $$(T_c-I)\Phi(s,t)=:\kappa(s)\psi(t).$$ The coefficients $\{\kappa_j\}$ are called the \emph{Frobenius constants} of $L$.
\end{defn}

\begin{rem}\label{R5.1}
The $\{\kappa_j\}$ were called ``Ap\'ery constants'' in the original version of \cite{BV}.  In our view this terminology is more appropriate for the values $\kappa(\ell)$, $\ell\in\ZZ\cap [1,d-1]$; see Remark \ref{R8.1} and Example \ref{E8.1}.  As this paper was in the finishing stages, the final version of [op. cit.] appeared in which the language of Definition \ref{D5.1} is used.
\end{rem}

The two Theorems that follow address (respectively) interpretation and computation of the Frobenius numbers.  The intervening Lemma gives a useful asymptotic description of the power-series coefficients of periods and related functions.

\begin{thm}\label{T5.1}
The first $n+1$ coefficients of $\kappa(s)^{-1}=:\sum_{i\geq 0}\alpha_i s^i$ yield the LMHS periods of Remark \ref{R4.2}.
\end{thm}

\begin{proof}
From $T_0 \sum_{j\geq 0} \phi_j s^j = e^{2\pi\ay s}\sum_{j\geq 0} \phi_j s^j$, we have $N_0\sum_{j\geq 0}\phi_j s^j = 2\pi\ay s\sum_{j\geq 0}\phi_j s^j = 2\pi\ay \sum_{j\geq 1}\phi_{j-1} s^j$ and thus $N_0 \phi_j = 2\pi\ay \phi_{j-1}$.  Writing $\e_n=\sum_{j=0}^n c_{n-j}\phi_j$ (for some constants $c_i$), applying $N_0$ repeatedly gives $(2\pi\ay)^{-k}\e_{n-k}=\sum_{j=0}^{n-k}c_{(n-k)-j}\phi_j$, hence $\e_{n-k}^{\text{an}}(0)=(2\pi\ay)^k c_{n-k}$.  Now 
\begin{align*}
\sum_{j\geq 0}\kappa_j \psi s^j &= (T_c -I)\sum_{j\geq 0} \phi_j s^j \;\implies \;\kappa_j \psi=(T_c-I)\phi_j \\
&\implies \sum_{j=0}^{\ell} c_{\ell-j}\kappa_j \psi = \sum_{j=0}^{\ell} c_{\ell-j}(T_c-I)\phi_j = \tfrac{1}{(2\pi\ay)^{n-\ell}}(T_c-I)\e_{\ell} \\
&\implies (2\pi\ay)^{n-\ell}(\Sigma_{j=0}^{\ell} c_{\ell-j}\kappa_j)\delta = (T_c-I)\ve_{\ell}=\bm{\delta}_{0\ell}\delta \\
&\implies \alpha_i = (2\pi\ay)^n c_i = (2\pi\ay)^i \e_i^{\text{an}}(0)\;\;\text{for }i=0,\ldots,n,
\end{align*}
as desired.
\end{proof}

\begin{lem}\label{L5.1}
Suppose a power-series $B(t)=\sum_{m\geq 0} B_m t^m$ with radius of convergence $|c|$ extends to an analytic function on $\widetilde{U^{\text{an}}}$, that the restriction of its modulus $|B(t)|$ \textup{(}or $|\int_0 B(t)dt|$\textup{)} to the cut disk $$\mathbb{D}_{\e}:=\{t\mid |t|<|c|+\e,\; \tfrac{t}{c}\notin [1,1+\tfrac{\e}{|c|})\}$$ is bounded \textup{(}for some $\e>0$\textup{)} by $\beta\in \RR_{>0}$, and that its monodromy satisfies $$\Lambda:=(T_c-I)B\sim\lambda (t-c)^{w-1}\;\;\;\text{near}\;\;\;t=c$$ for some $\lambda\in\CC^{\times}$ and $w\in\tfrac{1}{2}\ZZ_{\geq 2}$.  Then $$B_m \sim \frac{\lambda c^{w-1} \Gamma(w)}{2\pi\ay}\times \frac{1}{c^m m^w}  \;\;\;\text{as}\;\;\; m\to \infty.$$
\end{lem}

\begin{proof}
Write $\mathfrak{e}_m :=|c|(w+1)\tfrac{\log (m)}{m}$, and take $m\in\mathbb{N}$ sufficiently large that $\mathfrak{e}_m<\e$.  By Cauchy, we have $$2\pi\ay B_m = \int_{\partial\DD_{\mathfrak{e}_m}}\frac{B(t)}{t^{m+1}}dt=\oint_{|t|=|c|+\mathfrak{e}_m}\frac{B(t)}{t^{m+1}}dt+\int_c^{c(1+\frac{\mathfrak{e}_m}{|c|})}\frac{\Lambda(t)}{t^{m+1}}dt.$$ 
The first term's modulus is bounded by $$\tfrac{2\pi\beta}{(|c|+\mathfrak{e}_m)^m}=\tfrac{2\pi\beta}{|c|^m\left(1+(w+1)\frac{\log(m)}{m}\right)^m}  \sim \tfrac{2\pi\beta}{|c|^m m^{w+1}}=:B_m'.$$ The second term is asymptotic to 
\begin{align*}
\lambda &
\int_c^{c(1+\frac{\mathfrak{e}_m}{|c|})}\frac{(t-c)^{w-1}}{t^{m+1}}dt = 
\tfrac{\lambda c^{w-1}}{c^m}\sum_{j\geq 0}\tfrac{(-1)^j \binom{w-1}{j}}{m+j+1-w}
\underset{\sim 1}{\underbrace{ \{ 1-(1+\tfrac{\mathfrak{e}_m}{|c|})^{w-(m+j+1)} \}}} \\ 
&\sim \tfrac{\lambda c^{w-1}}{c^m}\int_0^1 X^{m-w} (1-X)^{w-1}dX = \tfrac{\lambda c^{w-1}}{c^m} \mathrm{B}(m-w+1,w)  \\ 
&\sim \tfrac{\lambda c^{w-1}}{c^m} \tfrac{\Gamma(w)}{m^w} =:B_m '',
\end{align*}
where the last line used Stirling's approximation for the beta function.   Since $\tfrac{B_m '}{|B_m ''|}\to 0$, we conclude that $2\pi\ay B_m\sim B_m''$.

If $B(t)$ is not bounded on $\DD_{\e}$, but $\int_0 B(t)dt=\sum_{m\geq 1}\tfrac{B_{m-1}}{m}t^m$ is (e.g. when $w=1$ and $B(t)\sim \tfrac{\lambda}{2\pi\ay}\log(t-c)$ as $t\to c$), then the argument gives $2\pi\ay\tfrac{B_{m-1}}{m}\sim \tfrac{\lambda c^w}{w c^m}\tfrac{\Gamma(w+1)}{m^{w+1}}$, which again gives $2\pi\ay B_m \sim B_m''$. 
\end{proof}

\begin{thm}\label{T5.2}
If $\M$ has strong conifold monodromy,\footnote{All we need is the consequence of Lemma \ref{L3.2}(ii).  The final revision of \cite{BV} includes a result of this form, but with much more restrictive conditions which the Lemma allows us to avoid.} then
\begin{enumerate}[label=\textup{(\roman*)}]
\item $\kappa(s) = c^s\, \lim_{k\to \infty} \frac{A_k(s)}{a_k}$, and thus 
\item $\kappa_m = \sum_{j=0}^m \frac{\log^j c}{j!}\, \lim_{k\to \infty}\frac{a_k^{(m-j)}}{a_k}.$
\end{enumerate}
\end{thm}

\begin{proof}
Observe that $\tilde{\Phi}:=\Phi-\kappa\phi_0$ has no monodromy about $t=c$ for any fixed $s=s_0$, so that $$\hat{\Phi}_{s_0}(t):=\Phi(s_0,t)-\tfrac{t^{s_0}}{c^{s_0}}\kappa(s_0)\phi_0(t)=\tilde{\Phi}(s_0,t)+(1-\tfrac{t^{s_0}}{c^{s_0}})\kappa(s_0)\phi_0(t)$$ has $(T_c-I)\hat{\Phi}_{s_0}=(1-\tfrac{t^{s_0}}{c^{s_0}})\kappa(s_0)\psi$.  The function
\begin{align*}
\B(t):&= t^{-s_0}\hat{\Phi}_{s_0} = \Phi^{\text{an}}(s_0,t)-\tfrac{\kappa(s_0)}{c^{s_0}}\phi_0(t) = \sum_{k\geq 0} A_k(s_0)t^k - \tfrac{\kappa(s_0)}{c^{s_0}}\sum_{k\geq 0}a_k t^k \\
&= \sum_{k\geq 0} (A_k(s_0)-\tfrac{\kappa(s_0)}{c^{s_0}} a_k) t^k,
\end{align*}
which is clearly invariant about $t=0$, then has $$(T_c -I)\B(t)=(t^{-s_0}-c^{-s_0})\kappa(s_0)\psi(t)\sim -\tfrac{s_0 \kappa(s_0)}{c^{s_0 +1}}(t-c)\psi(t)$$ for $t$ near $c$, while $(T_c-I)\phi_0(t) \sim \psi(t)$.  

By Lemma \ref{L3.2}(ii) we have $\psi(t)\sim C (t-c)^{\frac{n-1}{2}}$, as well as the boundedness of $\phi_0(t)=\sum_{m\geq 0}a_m t^m$ (or its integral) and $\B(t)=:\sum_{m\geq 0} b^{\B}_m t^m$ required for the application of Lemma \ref{L5.1}.  This yields
$$a_m\sim \tfrac{\mathtt{C}'}{c^m m^{\frac{n+1}{2}}}\;\;\;\;\text{and}\;\;\;\; b_m^{\B} \sim \tfrac{\mathtt{C}''}{c^m m^{\frac{n+3}{2}}}\, , $$
and so $\lim_{m\to \infty} \tfrac{b_m^{\B}}{a_m}=\tfrac{\mathtt{C''}}{\mathtt{C'}}\lim_{m\to \infty}\tfrac{1}{m}=0$.  That is, $$0=\lim_{m\to \infty} \frac{A_m(s_0)-\frac{\kappa(s_0)}{c^{s_0}}a_m}{a_m}=\lim_{m\to \infty}\left(\frac{A_m(s_0)}{a_m}-\frac{\kappa(s_0)}{c^{s_0}}\right)$$ which gives (i).  In fact, since $\tfrac{\mathtt{C}''}{\mathtt{C}'}=-\tfrac{n+1}{2}\tfrac{s_0 \kappa(s_0)}{c^{s_0}}$, this limit is uniform in $s$ in a neighborhood of $s=0$; we may thus expand $c^s$ and equate power-series coefficients, whence (ii).
\end{proof}

\begin{rem}\label{R5.2}
The flavor here is that, while $a_m$ and $A_m(s_0)$ have similar growth rate, the particular linear combination $A_m(s_0)-\tfrac{\kappa(s_0)}{c^{s_0}}a_m$ has somewhat slower growth.  This characterization of $\tfrac{\kappa(s_0)}{c^{s_0}}$ is vaguely reminiscent to that of $\zeta(3)$ in Ap\'ery's proof, though what happens at positive integer values of $s_0$ is much closer to the Ap\'ery phenomenon; see Remark \ref{R8.1} and Example \ref{E8.1}.
\end{rem}

\begin{example}\label{E5.1}
When $L$ is a hypergeometric operator (cf. \cite[\S3]{BV}), the results of this section suffice to \emph{compute} the matrix $\Omega_{\lim}$ from Remark \ref{R4.2}.  Suppose that $L$ arises as in \S\ref{S3}, with strong conifold monodromy at $c=1$, and takes the form $L=D^r + tP_1(D)$, with $P_1(D)=-\prod_{j=1}^r (D+\fa_j)$.  Then $q_0=1-t$ implies $\Sigma^{\times}=\{1\}$ and (via Prop. \ref{P6.1}(vi) below) $L^{\dagger}= L$, whence $\{\fa_j\}=\{1-\fa_j\}$ as sets and $\sum\fa_j=\tfrac{r}{2}$.  By the recurrence in Remark \ref{R4.1}, we have $$A_k(s)=\prod_{j=1}^r \frac{\Gamma(k+s+\fa_j)\Gamma(s+1)}{\Gamma(s+\fa_j)\Gamma(k+s+1)},$$ and so Theorem \ref{T5.2}(i) together with Stirling's formula yields $$\kappa(s)^{-1}=\lim_{k\to \infty}\frac{a_k}{A_k(s)}=\lim_{k\to \infty}\frac{A_k(0)}{A_k(s)}=\prod_{j=1}^r \frac{\Gamma(s+\fa_j)}{\Gamma(s+1)\Gamma(\fa_j)}.$$  This is enough to recover, for instance, the LMHSs for the complete intersection CY families in \cite{DM}, previously computed (using Iritani's mirror theorem \cite{Ir}) in \cite[\S4]{dSKP}.

To illustrate, consider the mirror quintic family ($\PP^4[5]$ in [op. cit.]),\footnote{Take $\vf=\sum_{i=1}^4 x_i+\prod_{i=1}^4 x_i^{-1}$ and replace $t$ by $t^5$ as at the end of Ex. \ref{E3.1}.} with $r=4$ and $\underline{\fa}=(\tfrac{1}{5},\tfrac{2}{5},\tfrac{3}{5},\tfrac{4}{5})$.  Taking the power-series expansion of $\prod_{j=1}^4 \tfrac{\Gamma(s+\frac{j}{5})}{\Gamma(s+1)\Gamma(\frac{j}{5})}$, we obtain ${}^t (\alpha_0,\alpha_1,\alpha_2,\alpha_3)=$ $${}^t\left(1,-5\log 5,10\zeta(2)+\tfrac{25}{2}\log^2 5,-40\zeta(3)-50(\log 5)\zeta(2)-\tfrac{125}{6}\log^3 5\right)$$ for the $0^{\text{th}}$ column of $\Omega_{\lim}$.  One arrives at the more standard form of this data by renormalizing the LMHS with respect to the local coordinate $\tfrac{t}{5^5}$, which means multiplying the column vector by $e^{(5\log 5)[N_0]_{\underline{\ve}}}$; this yields ${}^t (1,0,10\zeta(2),-40\zeta(3))$.  Moreover, the correct integral basis of the dual local system is not $\underline{\ve}=(\ve_0,\ve_1,\ve_2,\ve_3)$ but rather $(\ve_0,\ve_1,5\ve_2,5\ve_3)$; this leads us to multiply the last two entries of the vector by $5$.  The resulting invariants $50\zeta(2)$ and $-200\zeta(3)$ correspond exactly to $a=50$ and $b=-200$ in the table in [op. cit.].
\end{example}

\section{Conifold Gamma}\label{S6a}
The main theorem of \cite{BV}, a variant of which is given in the next section, is a precise relationship between $\kappa(s)$ and a specific Gamma function $\Gamma_c(s)$.  The latter involves particular choices of section $\mc\in H^0(U,\F^n)$ and homology class $\xi_c\in H_1(U,\MM(s)^{\vee}_{\QQ})$.  We first explain where the section comes from.

Let $\{e_j\}_{j=0}^n\subset \M^{\vee}(U)$ be the dual basis of $\{D^i\mu\}\subset \M(U)$.  Since the latter are meromorphic as sections of $\M_e$ on $\PP^1$, the former are meromorphic sections of $\M_e^{\vee}$.  Using $De_j + e_{j-1}=\tfrac{q_{r-j}}{q_0}e_n$, one checks as in \cite[\S4]{BV} that $L^{\dagger}(\tfrac{e_n}{q_0})=0$.  Moreover, by definition $e_n$ pairs to zero with generators of $\F^1\M$, and so it belongs to $\F^0\M^{\vee}=Q(\F^n\M)$, whence $e_n=\tfrac{Q(\mu)}{Y}$ for some $Y\in\CC(t)^{\times}$.  As $\langle e_n,D^n\mu\rangle=1$, $$Y=Y\langle e_n,D^n \mu\rangle=\langle Q(\mu),D^n\mu\rangle=Q(\mu,D^n\mu)$$ is the \emph{Yukawa coupling}.  Besides being a rational function, it has the following properties:

\begin{prop}\label{P6.1}
In the setting of \S\ref{S3}, we have:
\begin{enumerate}[leftmargin=1cm,label=\textup{(\roman*)}]
\item $Y(0)=\tfrac{1}{(2\pi\ay)^{n}\mathsf{Q}_0}.$
\item $DY=-\tfrac{2}{r}\tfrac{q_1}{q_0}Y$ \textup{(}recall $r=n+1$\textup{)}.
\item $p:=\tfrac{q_0 Y}{Y(0)}$ is a polynomial with $p(0)=1$.
\item The adjoint operator is given by $L^{\dagger}=\tfrac{1}{p}Lp$.
\item If $\M$ has strong conifold monodromy at $c$, then $p(c)\neq 0$.
\item The conditions $p\equiv 1$, $Y=\tfrac{Y(0)}{q_0}$, $L^{\dagger}= L$, and $q_1 = \tfrac{r}{2} Dq_0$ are equivalent.  They hold in particular when $|\Sigma^{\times}|=d$ and $\M$ has strong conifold monodromy at each point of $\Sigma^{\times}$.
\end{enumerate}
\end{prop}

\begin{proof}[Sketch]
(i) Applying $D^n$ to Lemma \ref{L3.2}(i) gives $\langle \ve_k,D^n\mu\rangle\sim (2\pi\ay)^{-n}\bm{\delta}_{kn}$ as $t\to 0$ hence $D^n\mu \sim (2\pi\ay)^{-n}\ve_n^{\vee}$.  So we have $(2\pi\ay)^n Y\sim Q(\mu,\ve_n^{\vee})\sim (-1)^n \langle Q(\ve_n^{\vee}),\mu\rangle =\langle \tfrac{\ve_0}{\mathsf{Q}_0},\mu\rangle \sim \tfrac{1}{\mathsf{Q}_0}$.

(ii) Take $m=\lfloor\tfrac{n}{2}\rfloor$. Applying $D$ to $Q(D^{i-1}\mu,D^{n-i}\mu)=0$ for $1\leq i\leq k$ yields $Q(D^k\mu,D^{n-k}\mu)=(-1)^k Y$; whence $DY=Q(D\mu,D^n \mu)+Q(\mu,-\Sigma_{i=0}^n \tfrac{q_{r-i}}{q_0}D^i\mu)=-mDY+(-1)^m Q(D^{m+1}\mu,D^{n-m}\mu)-\tfrac{q_1}{q_0}Y$, in which the middle term is $0$ for $n$ odd and $\tfrac{1}{2}DY$ for $n$ even.

(iii) At $\sigma\in\Sigma^{\times}$, $\ord_{\sigma}q_0\geq \rk(T_{\sigma}-I)=\rk(\mathrm{Res}_{\sigma}(\nabla))\geq -\ord_{\sigma}Y$.

(iv) Writing $L^{\dagger}$ and $\tfrac{1}{p}Lp$ in the form $\sum_i p_{r-i}(t) D^i$, they have the same $p_0$. But then they are equal because both kill $\tfrac{e_n}{q_0}=Q(\tfrac{\mu}{p})$: we have $Lp(\tfrac{e_n}{q_0})=Q(Lp\tfrac{\mu}{p})=Q(L\mu)=0$.

(v) Using $Y=\pm Q(D^m\mu,D^{n-m}\mu)$ from (ii) above with Lemma \ref{L3.2}(ii) shows that $Y$ has a simple pole at $t=c$; this cancels the zero of $q_0$.

(vi) The equivalence is clear. By (ii), $q_0$ has a zero at each zero or pole of $Y$, and $d$ strong conifolds exhausts the zeroes of $q_0$ (as $\deg(q_0)\leq d$).  So on $\PP^1\setminus \{\infty\}$, $Y$ has $d$ simple poles at these points, and no other zeroes or poles.
\end{proof}

Accordingly we shall set $$\mc:=\tfrac{1}{p}\mu\in H^0(U,\F^n\M)$$ and $\tilde{A}(t)=\sum_{m\geq 0}\tilde{a}_m t^m := \langle \ve_0,\mc\rangle = \tfrac{A(t)}{p(t)}$.  Notice that $\mc$ and thus $\tilde{A}$ are annihilated by $L^{\dagger}$. However, we also point out that the situation in (vi) is both easy to check and quite common for LG-models; and in that case, $\mc=\mu$ and $\tilde{A}=A$.

\begin{rem}\label{R6.1}
In view of Prop. \ref{P6.1}(iv), we say that $L$ is \emph{essentially self-adjoint} (cf. \cite[\S2.4]{vS}); this reflects the self-duality $\M^{\vee}\cong\M(n)$.  But the operator
$\hat{L}:=\tfrac{1}{\sqrt{p}}L\sqrt{p}$ satisfies $\hat{L}^{\dagger}= \hat{L}$, i.e. it is \emph{self-adjoint} on the nose.  Why don't we replace $L$ by this?  First, $p$ may not be a square, even for something as simple as a family of elliptic curves with an $\mathrm{I}_0^*$ fiber; in this case, $\hat{L}$ corresponds to a quadratic twist of $\M$ (not $\M$ itself).  Second, even if $p$ is a square, $\hat{L}$ corresponds to $\tfrac{\mu}{\sqrt{p}}$ (in place of $\mu$), which is a \emph{strictly meromorphic} section of $\M_e$ (unless of course $p\equiv 1$).  We prefer to work with the true \emph{Picard-Fuchs equation} of $\M$, i.e. the one corresponding to $\mu$ as we normalized it in \S\ref{S3}.

However, we feel obliged to point out that in the LG-model setting of Example \ref{E3.1}, $L$ itself turns out to be self-adjoint (i.e. $p\equiv 1$) with striking frequency.  Though one can certainly cook up counterexamples (e.g. see Remark \ref{R8.2}(ii)), consider the fact that this holds for all $23$ of the PF operators of order $3$ arising in the table of ``3D Minkowski period sequences'' in \cite{FANO}.  So the reader mainly interested in this case might consider ignoring the daggers from here on out.
\end{rem}

Turning to the homology class, we write $$P(x):=(x-1)^r =\sum_m \lambda_m x^m$$ and set
$$\xi_c := \left[ \sum \lambda_m \gamma_0^m \otimes \delta \otimes e^{2\pi\ay m s}+\gamma_c^{-1}\otimes \ve_0\otimes P(e^{2\pi\ay s}) \right]\in H_1(U,\MM(s)^{\vee}_{\QQ}).$$
This is well-defined since applying $\partial$ to the bracketed expression yields
\begin{align*}
\sum_m &\lambda_m \gamma_0^{-m}(\delta\otimes e^{2\pi\ay ms})-\sum_m \lambda_m \delta \otimes e^{2\pi\ay ms} + (\gamma_c -1)\ve_0 \otimes P(e^{2\pi\ay s}) \\
&= \sum_m \lambda_m \gamma_0^{-m}\delta \otimes 1 - \delta\otimes P(e^{2\pi\ay s}) + \delta\otimes P(e^{2\pi\ay s})\\
&= (\gamma_0^{-1}-1)^r \delta\otimes 1 = 0.
\end{align*}

\begin{defn}\label{D6.1}
The \emph{conifold Gamma} is $\Gamma_c(s):=\Gamma_{\xi_c,\mc}(s)$.
\end{defn}

Let $\fuo$ and $\fuc$ be neighborhoods of $0$ and $c$ containing $\gamma_0$ and $\gamma_c$ respectively (and no other roots of $q_0$); then $\fuo\cap U=\fuox:=\fuo\setminus \{0\}$ and $\fuc\cap U=\fucx:=\fuc\setminus \{c\}$, and $p\in \fuo\cap \fuc$.  Write $$\fu:= \fuox \cup \fuc \;\;\;\;\text{and}\;\;\;\; \fux :=\fuox\cup \fucx.$$
Notice that $\xi_c$ is supported on $\fux$.

\begin{prop}\label{P6.2}
Suppose $\M$ has strong conifold monodromy at $c$.  Then the $\QQ[e^{2\pi\ay s}]$-module of Gamma functions \textup{(}for $\mc$\textup{)} arising from $H_1(\fux,\MM(s)_{\QQ}^{\vee})$ has rank one and is spanned by $\xi_c$.
\end{prop}

\begin{proof}
By Proposition \ref{P6.1}(v), $\mc$ is a holomorphic section of $\F_e^n$ on $\fu$ (actually on $\fuo\cup \fuc$). Let $\cx_{\fu} \to \fu$ be the extension of $f^{-1}(\fux)\to \fux$ guaranteed by strong conifold monodromy, with nodal fiber over $c$; then $\mc\otimes \tfrac{dt}{t}$ belongs to $\Omega^{n+1}(\cx_{\fu})$, and so its pairing with $H_1(\fux,\MM(s)_{\QQ}^{\vee})$ factors through $\IH_1(\fu,\MM(s)_{\QQ}^{\vee})$. Since $H^0(\fu,\MM(s)_{\QQ})=\{0\}$, Euler-Poincar\'e says that the rank of $\IH^1(\fu,\MM(s)_{\QQ})$ (and its dual) is $(r-\rk(\MM(s)^{T_c}))-r \chi(\fu) = 1-0 =1$. Similarly, $\IH_1(\fuc,\MM(s)^{\vee}_{\QQ})=\{0\}$ and so $\IH_1(\fu,\MM(s)^{\vee}_{\QQ})\hookrightarrow \IH_1(\fu,\fuc;\MM(s)^{\vee}_{\QQ})\cong H_1(\fuox,\{p\};\MM(s)^{\vee}_{\QQ}) \overset{T_0-I}{\underset{\cong}{\to}}H_0(\{p\},\MM(s)_{\QQ}^{\vee})$ (where $T_0-I$ is an isomorphism thanks to the action on $t^s$).  The image of $\xi_c$ under the whole composition is just $\partial(\sum \lambda_m \gamma_0^m \otimes \delta \otimes e^{2\pi\ay m s})=-\delta\otimes P(e^{2\pi\ay s})$, which is certainly nonzero.
\end{proof}

\begin{rem}\label{R6.2}
Under the same hypothesis, for $\Re(s)>0$ we have that $\Gamma_c(s)=-P(e^{2\pi\ay s})\int_0^c \tfrac{\psi(t)}{p(t)}t^{s-1} dt$ \cite[Prop. 15]{BV}. However this is not particularly useful for computing the derivatives of $\Gamma_c$ at $s=0$ (which interest us below), since the corresponding integrals do not converge.  See Example \ref{E6.1} below for a small but amusing exception.
\end{rem}

\section{Gamma = kappa}\label{S6b}

Our main objective in this section is to present Theorem 30 of \cite{BV} in a more precise form that accounts for the self-duality of $\M$, relating the conifold Gamma for $\M$ to the kappa series for $L$.  The proof is similar to that in [op. cit.], but with sufficiently many changes that we summarize it here.

Let $\{\rho_i\}_{i=0}^n \subset \M(U)$ be the dual basis of $\{D^j (\tfrac{e_n}{q_0})\}_{j=0}^n$. Arguing as in \S\ref{S6a} (for $e_n$), $\rho_n$ belongs to $\F^n$ hence equals $F\mu$ for some $F\in \CC(t)^{\times}$.  To find it, write 
\begin{align*}
1&= \langle D^n(\tfrac{e_n}{q_0}),\rho_n\rangle = Q(D^n\tfrac{\mu}{q_0 Y},\rho_n) = \tfrac{(-1)^n}{q_0 Y}Q(\rho_n,D^n\mu) \\
&=\tfrac{(-1)^n F}{q_0 Y} Q(\mu,D^n\mu) = \tfrac{(-1)^n F}{q_0} \;\;\;\implies\;\;\;\rho_n=(-1)^n q_0\mu.
\end{align*}
Next, write $L^{\dagger}=\sum_{j=0}^r p_{r-j}(t) D^j$ (where $p_0=q_0$), and define 
\begin{align*}
\eta&\colon \co^{\text{an}} \to \M^{\vee,\text{an}} \;\;\;\text{by} \;\;\;\;\;\eta(\phi):=\textstyle{\sum}_{i=0}^n (D^i \phi)e_i\;\;\;\;\;\;\;\text{and}\\
\chi&\colon \co^{\text{an}} \to \M^{\text{an}} \;\;\;\;\;\text{by} \;\;\;\;\;\chi(\theta):=\tfrac{(-1)^n}{Y(0)}\textstyle{\sum}_{i=0}^n (D^i \tfrac{\theta}{p})\rho_i.
\end{align*}
Using $De_i+e_{i-1}=\tfrac{q_{r-i}}{q_0}e_n$ and (dually) $D\rho_i+\rho_{i-1}=\tfrac{p_{r-i}}{q_0}\rho_n = (-1)^n p_{r-i}\mu$, one easily computes that
$$D(\eta(\phi))=(L\phi)\tfrac{e_n}{q_0}\;\;\;\;\;\text{and}\;\;\;\;\; D(\chi(\theta))=(L^{\dagger}\tfrac{\theta}{p})\tfrac{\mu}{Y(0)}=(L\theta)\tfrac{\mu}{Y(0)p}.$$
Defining the bracket
$$[\;\;,\;\;]\colon \co^{\text{an}}\times \co^{\text{an}}\to \co^{\text{an}}\;\;\;\;\;\text{by}\;\;\;\;\;[\phi,\theta]:=\langle \eta(\phi),\chi(\theta)\rangle\,,$$
we have the crucial

\begin{lem}\label{L6.1}
\textup{(i)} $D[\phi,\theta]=\tfrac{1}{Y(0)p}\left\{\phi L\theta + (-1)^n \theta L \phi\right\}$.

\textup{(ii)} If $\alpha,\beta$ are local sections of $\MM_{\CC}^{\vee}$, with periods $\pi_{\alpha}=\langle \alpha,\mu\rangle$ and $\pi_{\beta}=\langle\beta,\mu\rangle$, then $\eta(\pi_{\alpha})=\alpha$, $Q(\chi(\pi_{\beta}))=\beta$, and $[\pi_{\alpha},\pi_{\beta}]=Q(\alpha,\beta)$.
\end{lem}

\begin{proof}
(i) follows immediately from $\langle D\eta(\phi),\chi(\theta)\rangle+\langle \eta(\phi),D\chi(\theta)\rangle=\tfrac{(-1)^n L\phi}{Y(0)}\sum_{i=0}^n (D^i\tfrac{\theta}{p})\langle \tfrac{e_n}{q_0},\rho_i\rangle+ \tfrac{L\theta}{Y(0)p}\sum_{i=0}^n (D^i\phi)\langle e_i,\mu\rangle$, since $\langle\tfrac{e_n}{q_0},\rho_i\rangle =\bm{\delta}_{i0}=\langle e_i,\mu\rangle$.  For (ii), notice that $L\pi_{\alpha}=0=L\pi_{\beta}$ $\implies$ $D(\eta(\pi_{\alpha}))=0=D(\chi(\pi_{\beta}))$ $\implies$ $\eta(\pi_{\alpha})$ and $Q(\chi(\pi_{\beta}))$ are sections of $\MM_{\CC}^{\vee}$.  To see \emph{which} sections, we pair them with $\mu$:  $\langle \eta(\pi_{\alpha}),\mu\rangle =\sum_{i=0}^n (D^i \pi_{\alpha})\langle e_i,\mu\rangle  = \pi_{\alpha}$; and $\langle Q(\chi(\pi_{\beta})),\mu\rangle = (-1)^n \langle Q(\mu),\chi(\pi_{\beta})\rangle = (-1)^n Y q_0 \langle \tfrac{e_n}{q_0},\chi(\pi_{\beta})\rangle = \pi_{\beta}$.  Hence $\langle \eta(\pi_{\alpha}),\chi(\pi_{\beta})\rangle = \langle \alpha, Q^{-1}(\beta)\rangle = Q(\alpha,\beta)$.
\end{proof}

\begin{thm}\label{T6.1}
In the setting of \S\ref{S3}, $$\kappa(s)=\frac{\mathsf{Q}_0}{\mathsf{Q}_c}\frac{(2\pi\ay)^n s^r}{(1-e^{2\pi\ay s})^r}\Gamma_c(s).$$ 
\end{thm}

\begin{proof}
Rewriting our representative of $\xi_c$ in the form $\sum_j \gamma_j \otimes \ve_j \otimes e^{2\pi\ay n_j s}$, we compute $$\mathscr{G}(s):=\sum_j e^{2\pi\ay n_j s}\int_{\gamma_j^{-1}}D[\e_j,\Phi]\frac{dt}{t}$$ in two different ways.  First, since $L\Phi=s^r t^s$ and $L\e_j=0$, $$D[\e_j,\Phi]=\frac{\e_j  s^r t^s}{Y(0)p}=(2\pi\ay)^n \mathsf{Q}_0 s^r \langle \ve_j ,\mc\rangle t^s$$ by Lemma \ref{L6.1}(i) and $\mathscr{G}(s)=(2\pi\ay)^n \mathsf{Q}_0 s^r \Gamma_c(s)$.  Second, by the FTC
\begin{align*}
\mathscr{G}(s) &= \textstyle{\sum}_j e^{2\pi\ay n_j s}(\gamma_j^{-1}-1)[\e_j,\Phi]\\
&=\textstyle{\sum}_m \lambda_m e^{2\pi\ay m s}(\gamma_0^{-m}-1)[\psi,\Phi]+P(e^{2\pi\ay s})(\gamma_c -1)[\e_0,\Phi]\\
&= \textstyle{\sum}_m\lambda_m e^{2\pi\ay ms}([\gamma_0^{-m}\psi,e^{-2\pi\ay m s}\Phi]-[\psi,\Phi])\\&\mspace{200mu}+P(e^{2\pi\ay s})([\e_0+\psi,\Phi+\kappa\psi]-[\e_0,\Phi]) \\
&= [P(\gamma_0^{-1})\psi,\Phi]-P(e^{2\pi\ay s})[\psi,\Phi]\\&\mspace{200mu}+P(e^{2\pi\ay s})[\psi,\Phi]+P(e^{2\pi\ay s})\kappa [\e_0+\psi,\psi]\\
&=P(e^{2\pi\ay s})\kappa(s)[\e_0+\psi,\psi]
\end{align*}
since $P(\gamma_0^{-1})=0$ on $\MM_{\QQ,p}^{\vee}$.  By Lemma \ref{L6.1}(ii), we have $[\e_0,+\psi,\psi]=Q(\ve_0+\delta,\delta)=Q(T_c\ve_0,(-1)^{n+1}T_c\delta)=(-1)^{n+1}Q(\ve_0,\delta)=(-1)^r \mathsf{Q}_c$.
\end{proof}

\begin{example}\label{E6.1}
Here is the simplest real example: let $\X\to \PP^1$ be the family of ``CY $0$-folds'' arising as in Example \ref{E3.1} from $\vf=-x+2-x^{-1}$, and $\M$ its reduced fiberwise $H^0$.  We have $L=D-4t(D+\tfrac{1}{2})=L^{\dagger}$, $\mathsf{Q}_0=2$, $\mathsf{Q}_c=-4$,  $c=\tfrac{1}{4}$, $A(t)=(1-4t)^{-\frac{1}{2}}=-\tfrac{1}{2}\psi(t)$, and (from Remark \ref{R6.2})
$\Gamma_c(s)=2(e^{2\pi\ay s}-1)\int_0^c A(t)t^{s-1}dt$. Applying Theorem \ref{T6.1} gives $\kappa(s)=s\int_0^{\frac{1}{4}}\tfrac{t^{s-1}dt}{\sqrt{1-4t}}=4^{-s} s \mathrm{B}(s,\tfrac{1}{2})=\tfrac{\Gamma(1+s)^2}{\Gamma(1+2s)}=\exp\left(2\sum_{k\geq 2} \tfrac{(-1)^{k-1}}{k}(2^{k-1}-1)\zeta(k)s^k\right)$.
\end{example}

\begin{cor}\label{C6.0}
Writing $L^{\dagger}=\sum_{i=0}^d t^i Q_i(D)$, the difference equation $\sum_{k=0}^d \tfrac{Q_k(-s-k)}{(s+k)^r}\kappa(s+k)=0$ holds.
\end{cor}

\begin{proof}
Divide Theorem \ref{T6.1} by $s^r$ and apply Theorem \ref{T2.1}.
\end{proof}

\begin{cor}\label{C6.1}
We have $\Gamma_c(0)=(-1)^r \tfrac{\mathsf{Q}_c}{\mathsf{Q}_0}2\pi\ay$; and for $m\in\ZZ_{>0}$, $\Gamma_c(m)=0$, $\Gamma_c(-m)=\Gamma_c(0)\tilde{a}_m$, and $\kappa(s)\sim \tfrac{(-m)^r}{(s+m)^r}\tilde{a}_m$ at $s=-m$.
\end{cor}

\begin{proof}
In addition to Theorem \ref{T6.1}, use Theorem \ref{T2.1} and Remark \ref{R2.2} (applied to $(L^{\dagger},\mc)$).
\end{proof}

The remarks that follow address the implications of Theorem \ref{T6.1} for the LMHS of $\M$ at $0$, whose periods turn out to be given by derivatives of (a variant of) the conifold Gamma at $s=0$.

\begin{rem}\label{R6.3}
Replacing $L$ by $L^{\dagger}$, $\mu$ by $\mc$, and $\psi$ by $\psi^{\dagger}:=\langle \delta,\mc\rangle$, we may define $\Phi^{\dagger}$ and $\kappa^{\dagger}$ as in Definitions \ref{D4.1} and \ref{D5.1}.  (Note that we are \emph{not} replacing $\M$ by $\M^{\vee}$.)  Then Theorem \ref{T5.1} remains true; and since $\mc(0)=\mu(0)$ in $\F^n_{e,0}$, we find that $\kappa_j^{\dagger}=\kappa_j$ for $j=0,\ldots,n$ (but not $j\geq r$).  Moreover, Theorem \ref{T6.1} and Corollary \ref{C6.1} hold replacing $\kappa$ by $\kappa^{\dagger}$, $\tilde{a}_m$ by $a_m$, and $\Gamma_c$ by $\Gamma_{\M}:=\Gamma_{\xi_c,\mu}$. (To see this, replace $[\e_j,\Phi]$ by $[\e_j,p\Phi^{\dagger}]$ in the proof.)  It follows that ${}^t (\kappa_0,\ldots,\kappa_n)$ is the product of a rational lower-triangular matrix by ${}^t (\tfrac{\Gamma_{\M}(0)}{2\pi\ay},\tfrac{\Gamma'_{\M}(0)}{(2\pi\ay)^2},\ldots,\tfrac{\Gamma_{\M}^{(n)}(0)}{(2\pi\ay)^{n+1}})$.

Now by Theorem \ref{T5.1}, ${}^t (\kappa_0,\ldots,\kappa_n)$ is the leading column of a period matrix for the dual of the LMHS of $\M$ at $0$.  As the LMHS of a polarized VHS is (up to twist) self-dual, we conclude that there exists a $\QQ$-basis $\{\mathpzc{e}_j\in W(N_0)_j\}_{j=0}^n$ of $\MM_{\QQ,p}$ such that $\mu(0)=\sum_{j=0}^n (2\pi\ay)^{-j-1}\Gamma^{(j)}_{\M}(0)\tilde{\mathpzc{e}}_j(0)$ in $\M_{e,0}$, where $\tilde{\mathpzc{e}}_j (t):=e^{-\frac{\log(t)}{2\pi\ay}N_0}\mathpzc{e}_j$.

Recall $\el(t):=\tfrac{\log(t)}{2\pi\ay}$.  Since $\tfrac{d^j}{dt^j} P(e^{2\pi\ay s})|_{s=0}=0$ for $j<r$, one finds that $$\frac{\Gamma_{\M}^{(j)}(0)}{(2\pi\ay)^{j+1}}=\sum_{m=0}^r \lambda_m \int_{\gamma_0^{-m}}\psi(t) \left(\el(t)+m\right)^j d\ell(t)\, ,$$ taking $\log(p)\in\RR$ at the start of each path.  As a formula for actually computing the LMHS this seems closely related to the ``Cauchy integral method'' in \cite[\S5]{dSKP}, though more unwieldy.  Rather, its importance is theoretical, as the next Remark demonstrates.
\end{rem}

\begin{rem}[Limiting motive]\label{R6.4}
The family of relative motives $L_t :=(\GG_m,\{1,t\})$ underlies the rank-$2$ connection $\D/\D D^2$ in $\GG_m$, with periods $1$ and $\el(t)$ over the cycles $S^1$ and $[1,t]$ in $H_1(L_t)$.  Write $\M[n]$ for the VMHS $\M\otimes \mathrm{Sym}^n H^1(L_t)$ on $U$, and $\Xi_j\in H_1(U^m,\MM[n]^{\vee}_{\QQ})$ for the class of the cycle $\sum_m \lambda_m \gamma_0^m\otimes \delta\otimes ([0,t]+mS^1)^j(S^1)^{n-j}$ (closed for $j<r$). Putting $\varpi :=\mu\otimes (\tfrac{dz_1}{2\pi\ay z_1}\wedge \cdots \wedge \tfrac{dz_n}{2\pi\ay z_n})\otimes \tfrac{dt}{2\pi\ay t}\in H^1_{\mathrm{dR}}(U,\M[n])$, for $0{\leq} j{\leq} n$ we recover $(2\pi\ay)^{-j-1}\Gamma_{\M}^{(j)}(0)$ as periods $\langle \Xi_j,\varpi\rangle$ of the connection.

These are also periods of a relative variety.  Inside our smooth total space $\X\overset{f}{\to}\PP^1$, consider $\X_{\mathbb{G}_m}:=f^{-1}(\mathbb{G}_m)$. Let $\DD[n]\subset \GG_m^n\times \GG_m$ be the divisor defined by $\prod_{i=1}^n (z_i-1)(z_i-t)$, and write $\mathrm{X}[n]:=\X_{\mathbb{G}_m}\times \GG_m^n$, $\mathrm{D}[n]:=\X_{\mathbb{G}_m}\times_{\GG_m}\DD[n]$, and $\mathrm{X}[n]_{\text{rel}}:=(\mathrm{X}[n],\mathrm{D}[n])$.  Then recalling that $\mu$ is a \emph{holomorphic} section of $\F_e^n$, we may regard $\Xi_j$ and $\varpi$ as classes in $H_{2n+1}(\mathrm{X}[n]_{\text{rel}},\QQ)$ and $F^{2n+1}H^{2n+1}(\mathrm{X}[n]_{\text{rel}},\CC)$ respectively.  A further refinement is obtained by observing that $\IH^1(\GG_m,\MM[n])$ yields a sub-MHS/motive of $H^{2n+1}(\mathrm{X}[n]_{\text{rel}})$, of which the $\langle \Xi_j,\varpi\rangle$ remain periods.

Now in general these are only some of the periods, not all of the periods, of this MHS.  (Alas, the part of $\mathrm{X}[n]_{\text{rel}}$ over $\fu$ is not a motive.)  But there \emph{is} a case in which the $\{\Xi_j\}$ \emph{span} $\IH_1(\GG_m,\MM[n])$, and that is when $|\Sigma^{\times}|=1$:  indeed, by Euler-Poincar\'e we find that $\rk(\IH^1(\GG_m,\MM[n]))=n+1$.  So in this ``hypergeometric'' case, we obtain a \emph{motive with Hodge realization equal to the LMHS of $\M$ at $t=0$}.

Naturally, we have left aside the messiness of constructing a log-resolution of $\left(\overline{\mathrm{X}[n]},\overline{\mathrm{D}[n]}\cup(\overline{\mathrm{X}[n]}\setminus\mathrm{X}[n])\right)$ and the required projectors, but it is clear that this can be done.  Moreover, despite various ``limiting motive'' constructions, this is the first of which we are aware \emph{with the desired Hodge realization} outside of the weight-one setting \cite{Hao}, further illustrating the power of the approach of Bloch and Vlasenko.
\end{rem}

\section{The unipotent extensions}\label{S7}

Closely related to the Frobenius deformation in \S\ref{S4} is an inverse limit of VMHSs whose periods are annihilated by $D^m L(\cdot)$ for some $m$ \cite[\S5]{BV}.  Our initial intention in this section was to investigate these VMHSs, but (given our choice of $\mu$ and thus $L$) it turns out to be more natural to consider $D^m L^{\dagger}$, essentially because the periods of \emph{its} adjoint $L D^m$ integrate the periods of $\mu$.  The warning here is that while $L$ and $L^{\dagger}$ define isomorphic $\D$-modules, $D^m L$ and $D^m L^{\dagger}$ do \emph{not} -- unless, of course, $L^{\dagger}=L$.

Fix $m\in \ZZ_{> 0}$, and consider the connection
$$0\to \K\to \E\overset{\pi}{\to} \M \to 0$$
on $U$ given by $\D/\D D^m \underset{L^{\dagger}(\cdot)}{\hookrightarrow} \D/\D D^m L^{\dagger} \twoheadrightarrow \D/\D L^{\dagger}$.  The dual sequence
$$0\to \M^{\vee} \to \E^{\vee} \to \K^{\vee} \to 0$$
is given by $\D/\D L \underset{D^m(\cdot)}{\hookrightarrow} \D/\D L  D^m \twoheadrightarrow \D/\D D^m $, and the solution sheaves by
$$ 0 \to \MM^{\vee}_{\CC}\overset{\imath}{\to }\EE^{\vee}_{\CC} \to \KK^{\vee}_{\CC}\to 0.$$
Via $\imath$, the basis $\ve_0,\ldots,\ve_n$ of $\MM^{\vee}_{\QQ,p}$ may be regarded as elements of $\EE_{\CC,p}^{\vee}$.  Let $\Omega \in \E(U)$ denote the image of $1\in \D/\D D^m L^{\dagger}$, so that $\pi(\Omega)=\mc$.

\begin{defn}\label{D7.0}
The connection $\E$ (or its restriction to a subset of $U$) \emph{underlies a $\QQ$-VMHS} if there is a $\QQ$-local system $\EE_{\QQ}\subset \EE_{\CC}=\ker(\nabla)$ with $\EE_{\QQ}\otimes \CC \cong \EE_{\CC}$, a flag $\F^{\bullet}\subset \E$ of holomorphic sub-bundles with $D\F^{\bullet}\subset \F^{\bullet-1}$, and a weight filtration $\mathcal{W}_{\bullet}$ on $\EE_{\QQ}$, such that the pointwise restrictions of $(\EE_{\QQ},\mathcal{W}_{\bullet},\F^{\bullet}\E)$ define $\QQ$-MHSs.
\end{defn}

Here we shall mainly be concerned with the restriction of $\E$ to the punctured neighborhood $\fuox$ and (provided this underlies a VMHS) its LMHS at $0$, in which $\mathcal{W}_{\bullet}$ is replaced by the relative monodromy weight filtration $W(N_0,\mathcal{W})_{\bullet}$ (whose existence is not an issue here).

\begin{thm}\label{T7.1}
$\E|_{\fuox}$ underlies a $\QQ$-VMHS which is the \textbf{unique} one on
\noindent\begin{minipage}{0.33\textwidth}
\begin{tikzpicture} [scale=0.4]
\draw [thick,->] (0,-4.5) -- (0,5);
\node [] at (0.5,5) {$q$};
\draw [thick,->] (-4.5,0) -- (5,0);
\node [] at (5,-0.5) {$p$};
\node [blue] at (-3,3) {\huge $\E$};
\draw [gray] (-1,-0.3) -- (-1,0.3);
\draw [gray] (-4,-0.2) -- (-4,0.2);
\draw [gray] (1,-0.2) -- (1,0.2);
\draw [gray] (4,-0.2) -- (4,0.2);
\draw [gray] (-0.2,-1) -- (0.2,-1);
\draw [gray] (-0.2,-4) -- (0.2,-4);
\draw [gray] (-0.2,1) -- (0.2,1);
\draw [gray] (-0.2,4) -- (0.2,4);
\node [gray] at (4,-0.5) {\small $n$};
\node [gray] at (-0.5,4) {\small $n$};
\node [gray] at (-1,0.5) {\tiny $-1$};
\node [gray] at (0.7,-0.9) {\tiny $-1$};
\node [gray] at (-4,0.5) {\tiny $-m$};
\node [gray] at (0.7,-3.9) {\tiny $-m$};
\filldraw [blue] (0,4) circle (4pt);
\filldraw [blue] (1,3) circle (4pt);
\filldraw [blue] (1.9,2.1) circle (1pt);
\filldraw [blue] (2.5,1.5) circle (1pt);
\filldraw [blue] (3.1,0.9) circle (1pt);
\filldraw [blue] (4,0) circle (4pt);
\filldraw [blue] (-1,-1) circle (4pt);
\filldraw [blue] (-4,-4) circle (4pt);
\filldraw [blue] (-1.9,-1.9) circle (1pt);
\filldraw [blue] (-2.5,-2.5) circle (1pt);
\filldraw [blue] (-3.1,-3.1) circle (1pt);
\draw [thick,->] (0,-16.5) -- (0,-7);
\node [] at (0.5,-7) {$q$};
\draw [thick,->] (-4.5,-12) -- (5,-12);
\node [] at (5,-12.5) {$p$};
\node [blue] at (-3,-9) {\huge ${\E}_{\scriptscriptstyle\text{lim}}$};
\draw [gray] (-1,-12.3) -- (-1,-11.7);
\draw [gray] (-4,-12.2) -- (-4,-11.8);
\draw [gray] (1,-12.2) -- (1,-11.8);
\draw [gray] (4,-12.2) -- (4,-11.8);
\draw [gray] (-0.2,-13) -- (0.2,-13);
\draw [gray] (-0.2,-16) -- (0.2,-16);
\draw [gray] (-0.2,-11) -- (0.2,-11);
\draw [gray] (-0.2,-8) -- (0.2,-8);
\node [gray] at (4,-12.5) {\small $n$};
\node [gray] at (-0.5,-8) {\small $n$};
\node [gray] at (-1,-11.5) {\tiny $-1$};
\node [gray] at (0.7,-12.9) {\tiny $-1$};
\node [gray] at (-4,-11.5) {\tiny $-m$};
\node [gray] at (0.7,-15.9) {\tiny $-m$};
\filldraw [blue] (4,-8) circle (4pt);
\filldraw [blue] (3,-9) circle (4pt);
\draw [red,thick,->] (3.8,-8.2) -- (3.2,-8.8); 
\node [red] at (4,-9) {$N_0$};
\filldraw [blue] (2.33,-9.66) circle (1pt);
\filldraw [blue] (2,-10) circle (1pt);
\filldraw [blue] (1.66,-10.33) circle (1pt);
\draw [red,thick,->] (0.8,-11.2) -- (0.2,-11.8);
\filldraw [blue] (1,-11) circle (4pt);
\filldraw [blue] (0,-12) circle (4pt);
\filldraw [blue] (-1,-13) circle (4pt);
\draw [red,thick,->] (-0.2,-12.2) -- (-0.8,-12.8);
\filldraw [blue] (-3,-15) circle (4pt);
\filldraw [blue] (-4,-16) circle (4pt);
\draw [red,thick,->] (-3.2,-15.2) -- (-3.8,-15.8);
\filldraw [blue] (-1.66,-13.66) circle (1pt);
\filldraw [blue] (-2,-14) circle (1pt);
\filldraw [blue] (-2.22,-14.33) circle (1pt);
\end{tikzpicture}
\end{minipage}
\begin{minipage}{0.67\textwidth}
$\fuox$ with underlying $\D$-module $\D/\D D^m L^{\dagger}$ and having the properties\textup{:}
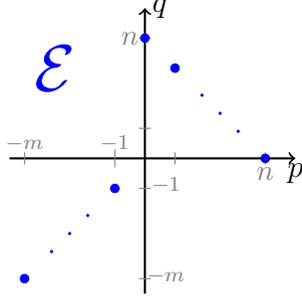
\begin{enumerate}[leftmargin=1cm,label=\textup{(\roman*)}]
\item $\Omega$ belongs to $\F^n$\textup{;}
\item $\imath(\ve_0)$ belongs to $\EE^{\vee}_{\QQ,p}$\textup{;}
\item $\EE_{\QQ}^{(\vee)}$ extends to $\fu$ \textup{(}i.e. is closed under $T_c$\textup{);} and
\item $\rk(\E^{k,n-k})=1$ for $0\leq k \leq n$, $\rk(\E^{-k,-k})=1$ for $1\leq k\leq m$, and all other $\E^{p,q}$ are zero.
\end{enumerate}\vspace{3mm}
This VMHS satisfies, in addition, the following\textup{:}
\begin{enumerate}[leftmargin=1cm,label=\textup{(\alph*)}]
\item $\pi|_{\fuox}$ is a morphism of $\QQ$-VMHS\textup{;}
\item $(T_c-I)\EE_{\QQ}^{\vee}\subset \QQ\imath(\delta)$\textup{;}
\item the LMHS $\E_{\lim}$ of $\E$ at $0$ is Hodge-Tate, with $N_0^{n+m}\neq 0$\textup{;} and
\item the first $n+m+1$ power-series coefficients of $\kappa^{\dagger}(s)^{-1}$ yield the LMHS periods at $0$ \textup{(}extending Theorem \ref{T5.1}/Remark \ref{R6.3}\textup{)}.
\end{enumerate}
\end{minipage}
\end{thm}

\begin{proof}
The Hodge filtration $\F^{n-k}\E=\mathcal{O}\langle \Omega,D\Omega,\ldots,D^k\Omega\rangle$, as well as the weight filtration $\mathcal{W}_n\E=\E$, $\mathcal{W}_{n-1}\E=\mathcal{W}_{-2}\E=\K$, $W_{-2n+2k}\E=\mathcal{O}\langle D^{m-k-1} L^{\dagger}\Omega,\ldots,D^{m-1}L^{\dagger}\Omega\rangle$ ($k=1,\ldots,m$), are forced upon us by transversality, $D^m L^{\dagger}\E=\{0\}$, and (iv).  Sending $\Omega\mapsto\mc$ projects $(\E,\F^{\bullet})\twoheadrightarrow (\M,\F^{\bullet})$.  We need to construct the $\QQ$-local system and show that $\mathcal{W}_{\bullet}$ is compatible with the resulting $\QQ$-structure; this will be carried out on the dual.

Writing $\Phi^{\dagger}=\sum_k \phi_k^{\dagger}s^k$ (cf. Remark \ref{R6.3}), we find exactly as in \S\ref{S5} that $(T_c-I)\phi_k^{\dagger}=\kappa_k^{\dagger}\psi^{\dagger}$, $N_0 \phi_k^{\dagger}=2\pi\ay\phi_{k-1}^{\dagger}$, and $L^{\dagger}\phi_{n+j}^{\dagger}=\tfrac{\log^{j-1}t}{(j-1)!}$ $\implies$ $D^j L^{\dagger}\phi_{n+j}=0$.  In particular, this yields identifications 
$$\MM^{\vee}_{\CC,p}\overset{\cong}{\underset{\langle\cdot,\mc\rangle}{\to}}\mathrm{Sol}_p(L^{\dagger})=\CC\langle \phi_0^{\dagger},\ldots,\phi_n^{\dagger}\rangle$$ and 
$$\EE_{\CC,p}^{\vee}\overset{\cong}{\underset{\langle\cdot,\Omega\rangle}{\to}} \mathrm{Sol}_p(D^m L^{\dagger})=\CC\langle \phi_0^{\dagger},\ldots,\phi_{n+m}^{\dagger}\rangle$$ 
for the $\CC$-local systems.  Omitting ``$\imath(\cdot)$'' for simplicity, we must extend the $\QQ$-basis $\{\ve_0,\ldots,\ve_n\}$ of $\MM_{\QQ,p}^{\vee}$ by some $\ve_{n+1},\ldots,\ve_{n+m}\in \EE^{\vee}_{\QQ,p}$.  Recalling from the proof of Theorem \ref{T5.1} (with daggers inserted) that $\e_k^{\dagger}=(2\pi\ay)^{-k}\sum_{j=0}^k \alpha^{\dagger}_{k-j}\phi_j^{\dagger}$ for $k=0,\ldots,n$, we can simply use this formula to \emph{define} $\e_k^{\dagger}$ and $\ve_k:=\langle\cdot,\Omega\rangle^{-1}(\e_k^{\dagger})$ for $k=n+1,\ldots,n+m$.  Then we automatically get $N_0\ve_k=\ve_{k-1}$, and $$(T_c-I)\ve_k = \textstyle{\left( \tfrac{1}{(2\pi\ay)^k}\sum_{j=0}^k \alpha^{\dagger}_{k-j} \kappa^{\dagger}_{j}\right)}\delta = \left\{\begin{matrix}\delta,& k=0\\ 0,& k>0.\end{matrix}\right.$$ The LMHS periods are just the $(2\pi\ay)^k \e_k^{\dagger,\text{an}}(0)=\sum_{j=0}^k \alpha^{\dagger}_{k-j} \phi^{\dagger,\text{an}}_{j}(0)=\sum_{j=0}^k \alpha^{\dagger}_{k-j} \bm{\delta}_{0j}=\alpha^{\dagger}_k$.  The weight filtration dual to $\mathcal{W}_{\bullet}$ may be described as $\mathcal{W}^{\vee}_{-n}=\mathcal{W}^{\vee}_{1}=\MM_{\QQ}^{\vee}$ and $\mathcal{W}_{2k}^{\vee}=\mathcal{W}_{2k+1}^{\vee}=\MM_{\QQ}^{\vee}+\QQ\langle \ve_{n+1},\ldots,\ve_{n+k}\rangle$ (the point being that it kills $\mathcal{W}_{-2k-2}\E=\mathcal{O}\langle D^k L^{\dagger}\Omega,\ldots,D^{m-1}L^{\dagger}\Omega\rangle$ because $\mathcal{W}_{2k}^{\vee}=\text{im}(N_0^{m-k})$ while $\langle N_0^{m-k}(\cdot),D^{\geq k} L^{\dagger}\Omega\rangle=\langle (\cdot),D^{\geq m}L^{\dagger}\Omega\rangle=0$).  This completes the proof of existence of the $\QQ$-VMHS and properties (a)-(d).

For uniqueness, suppose another $\hat{\E}$ satisfies (i)-(iv).  Again $\F^{\bullet}$ and $\mathcal{W}_{\bullet}$ are forced upon us, so that $\hat{\E}$ and $\E$ are the same as bifiltered $\D$-modules.  To show $\hat{\EE}_{\QQ,p}^{\vee}=\EE_{\QQ,p}^{\vee}$ inside $\EE^{\vee}_{\CC,p}$, write $E_k:=\ker(N_0^k)\subset \EE^{\vee}_{\CC,p}$ and assume inductively $E_{k-1}\cap \EE_{\QQ,p}^{\vee}=E_{k-1}\cap \hat{\EE}_{\QQ,p}^{\vee}$ (with (ii) providing the ``base case'' $k=1$).  We have an isomorphism\footnote{We are not \emph{using} (a)-(c) here (as we must not!), only (ii)-(iii) and the differential equation $D^m L^{\dagger}(\cdot)=0$.  Since the latter is essentially $D^{m+n+1}$ at $0$, and $N_0=-2\pi\ay \mathrm{Res}_0(\nabla_D)$, we get $N_0^{n+m}\neq 0$ directly.  We saw at the beginning of \S\ref{S5} that $(T_c-I)\EE^{\vee}_{\CC,p}\subset \CC\delta$ (from the differential equation only).  The map is an isomorphism because we have $E_0=\ker(N_0)\underset{T_c-I}{\overset{\cong}{\to}}\CC\imath(\delta)$ by our earlier assumptions on $\MM_{\CC}^{\vee}$ in \S\ref{S3}.} $$(N_0,T_c-I)\colon E_k \overset{\cong}{\to} E_{k-1}\oplus \CC\imath(\delta),$$ under which any choice of $\QQ$-structure on the left-hand side consistent with (iii) must go to $(E_{k-1}\cap \EE_{\QQ,p}^{\vee})\oplus \QQ\imath(\delta)$ on the right.  So $E_k\cap \EE_{\QQ,p}^{\vee}=E_k\cap \hat{\EE}_{\QQ,p}^{\vee}$.
\end{proof}

\begin{cor}\label{C7.1}
Given a $\QQ$-VMHS $\,\E'$ over $U$ of type \textup{(iv)}, with a surjective morphism to the $\QQ$-VHS $\M$ sending $\omega \in H^0(U,\F^n\E')$ to $\mc$, and $D^m L^{\dagger}\omega=0$.  Then $\E'|_{\fuox}\cong \E_{\textup{Thm. }\ref{T7.1}}$ as a $\QQ$-VMHS, and in particular \textup{(b)} resp. \textup{(c)-(d)} hold for ${\EE'}_{\QQ}^{\vee}$ resp. $\E_{\lim}$.
\end{cor}

\begin{proof}
Clearly (i)-(iii) are immediate from the hypotheses.
\end{proof}

There is a plentiful source of such $\QQ$-VMHS in the case $m=1$.  Let $\vf$ be a reflexive Laurent polynomial.  With notation as in Example \ref{E3.1}, and writing $\cx^{\times}:=\beta^{-1}(\GG^n_m)$, we can take the cup-product of the $\beta^* x_i\in \co^{\times}(\cx^{\times})\cong H^1_{\mathrm{M}}(\cx^{\times},\QQ(1))$ ($i=1,\ldots,n+1$) to get a motivic cohomology class $\{\ux\}\in H^{n+1}_{\mathrm{M}}(\cx^{\times},\QQ(n+1))$ called the \emph{coordinate symbol}.

\begin{defn}\label{D7.1}
We say that $\vf$ is \emph{tempered} if $\{\ux\}$ extends to a class $\bx\in H^{n+1}_{\mathrm{M}} (\cx\setminus X_0,\QQ(n+1))$. (One may assume without loss of generality that $\vf\in\bar{\QQ}[x_1^{\pm 1},\ldots,x_{n+1}^{\pm1}]$, since --- up to scale --- this is a necessary condition for temperedness \cite[Prop. 4.16]{DK}. Minkowski polynomials are expected to be tempered in general; this is known for $n\leq2$ \cite{dS}.  See \cite[\S3]{DK} for further discussion.)
\end{defn}

Recall that a (graded-polarizable) $\QQ$-VMHS $\V$ on $U$ is called \emph{admissible} (with respect to $\PP^1$) if it is the restriction of a polarizable mixed Hodge module from $\PP^1$.  Admissibility always holds for geometric variations, and guarantees that a LMHS exists at each $\sigma\in\Sigma$; henceforth these are written $\psi_{\sigma}\V$.\footnote{The LMHS is only well-defined with a choice of local parameter vanishing to first order at $\sigma$, and this parameter would usually be written as the subscript; for us, the parameter is always $t-\sigma$ ($\sigma$ finite) or $t^{-1}$ ($\sigma=\infty$).}

\begin{defn}\label{D7.2}
An admissible VMHS of the form $$0\to \H \to \V \to \QQ_U(0)\to 0,$$ where $\H$ is a $\QQ$-PVHS on $U$, is called an \emph{admissible normal function}; we write $\V\in \mathrm{ANF}(\H)$.  (These are only interesting, i.e. can be non-split, for $\H$ of weight $\leq -1$.  If the weight is $<-1$, they are called \emph{higher normal functions} since Bloch's higher Chow groups, or equivalently motivic cohomology, are the standard source.)  Using $\mathrm{Ext}^1_{\text{MHS}}(\QQ(0),\HH_t)\cong\HH_{t,\CC}/\left(F^0 \HH_{t,\CC}+\HH_{t,\QQ}\right)$, pointwise restriction yields a holomorphic section $\V_t$ of the generalized Jacobian bundle $J(\H):=\H/\left(F^0\H + \HH_{\QQ}\right)$; it is in this sense that $\V$ is a ``function''.
\end{defn}

If $\vf$ is tempered and good (cf. Example \ref{E3.1}), we may construct a higher normal function by applying the composition
\begin{align*}
H^{n+1}_{\mathrm{M}}(\X_U,\QQ(n+1))\overset{c_{\mathscr{H}}}{\to} \;& H_{\mathscr{H}}^{n+1}(\X_U,\QQ(n+1))\\ &\cong \mathrm{Ext}^1_{\text{MHM}(\X_U)_{\X}^{\text{ps}}}(\QQ_{\X_U}(0),\QQ_{\X_U}(n+1)) 
\\
\overset{\mathrm{Gr}^1_{\mathcal{L}}}{\twoheadrightarrow} \mathrm{Ext}^1_{\text{AVMHS}(U)}(\QQ_U(0),&\H_f^n (n+1)) \twoheadrightarrow \mathrm{Ext}^1_{\text{AVMHS}(U)}(\QQ_U(0),\M(n+1))
\end{align*}
of the absolute-Hodge cycle-class map \cite{KLe}, the projection to the bottom nonzero Leray-graded piece, and the projection from $\H_f^n$ to its direct summand $\M$.  The corresponding section of $J(\M(n+1))\cong \M/\MM(n+1)$ is evaluated at $t\in U$ by applying 
\begin{align*}
\mathrm{AJ}\colon H_{\mathrm{M}}^{n+1}(X_t,\QQ(n+1))\to \;&\mathrm{Ext}^1_{\text{MHS}}(\QQ(0),H^n_t(n+1))\\&\cong H^n(X_t,\CC/\QQ(n+1))
\end{align*}
to $\bx_t:=\imath^*_{X_t}\bx$ and projecting to $\MM_{t,\QQ}\otimes \CC/\QQ(n+1)$.

\begin{defn}\label{D7.3}
This higher normal function, written $$\V_{\vf}\in \mathrm{ANF}(\M(n+1)),$$ is called the \emph{box extension} associated to $\vf$.
\end{defn}

\begin{thm}\label{T7.2}
If $\M$ arises from a reflexive, good, tempered Laurent polynomial $\vf$, then the dual box extension provides a geometric realization of the unipotent extension with $m=1$: $$\V_{\vf}^{\vee}(1)|_{\fuox}\cong \E_{\textup{Thm. }\ref{T7.1}}.$$ Consequently, the periods of $\psi_0 \V_{\vf}^{\vee}$ and $\psi_0\V_{\vf}$ are given by $\{\alpha_0^{\dagger},\ldots,\alpha_{n+1}^{\dagger}\}$ and $\{\kappa_0^{\dagger},\ldots,\kappa_{n+1}^{\dagger}\}$ respectively.
\end{thm}

\begin{proof}
Since $\V_{\vf}$ is an extension of $\QQ_U(0)$ by $\M(n+1)\cong \M^{\vee}(1)$, $\V_{\vf}^{\vee}(1)$ is an extension of $\M$ by $\QQ(1)$, and is of the form (iv), with dual maps $\pi\colon \V_{\vf}^{\vee}(1)\to\M$ and $\imath\colon \MM^{\vee}\to \VV_{\vf}(-1)$. Let $\omega \in H^0(U,\F^n\V_{\vf}^{\vee}(1))$ be the unique section mapping to $\mc$.  We must show that $DL^{\dagger}$ annihilates all periods of $\omega$.  Clearly $\langle \imath(\ve_j),\omega\rangle=\langle \ve_j,\pi(\omega)\rangle=\langle\ve_j,\mc\rangle$ is killed by $L^{\dagger}$ for $j=0,\ldots,n$; so it remains to check that the remaining independent period (which will not be killed by $L^{\dagger}$) is killed by $DL^{\dagger}$.

Let $\rt_F\in H^0(U,\F^0\V_{\vf})$ and $\rt_{\QQ}\in H^0(\widetilde{U^{\text{an}}},\VV_{\vf,\QQ})$ be sections mapping to $1\in \QQ(0)$; their difference $\rt=\rt_{\QQ}-\rt_F$ is a multivalued section of $\M$ whose image in $J(\M(n+1))$ ``is'' $\V_{\vf}$ (as a normal function).  By \cite[Cor. 4.1]{DK} we have $D\rt=(2\pi\ay)^n\mu$.\footnote{The proof there is long and uses regulator currents; here is a sketch of a more hands-off proof: we can go from $H_{\mathscr{H}}(\X_U,\QQ(n+1))$ to $H^1_{\text{dR}}(U,\M)$ by (a) mapping to $\mathrm{Hom}_{\text{MHS}}(\QQ(0),H^{n+1}(\X_U,\QQ(n+1)))$ and taking the first Leray graded piece, or (b) taking fiberwise restrictions to get a section of $J(\M(n+1))$ and applying $\nabla$.  It is a standard exercise to show that these two compositions are equal; and (a) is given by $\mathrm{dlog}(\ux)=(2\pi\ay)^n\mu\otimes \tfrac{dt}{t}$, while (b) is exactly $\nabla \rt=D\rt\otimes \tfrac{dt}{t}$.} This implies that $D^k\rt\in \F^{n+1-k}$ for $k>0$ so that $Q(D^k\rt,\mc)=0$ for $0<k<n+1$.

Now consider the (holomorphic, multivalued) truncated higher normal function $$V_{\vf}(t):=Q(\rt,\mc),$$ and calculate
\begin{align*}
L^{\dagger}&= q_0 Q(D^{n+1}\rt,\mc)+Q(\rt,\cancelto{\scriptstyle{0}}{L^{\dagger}\mc})\\
&=(2\pi\ay)^n q_0 Q(D^n\mu,\tfrac{\mu}{p})=\tfrac{(2\pi\ay)^n q_0}{p}(-1)^n Q(\mu,D^n\mu)\\
&= \tfrac{(-2\pi\ay)^n q_0 Y}{Y(0)^{-1}q_0 Y}=\tfrac{(-2\pi\ay)^n}{(2\pi\ay)^n \mathsf{Q}_0}=\tfrac{(-1)^n}{\mathsf{Q}_0}.
\end{align*}
On the other hand, the duality pairing $\V_{\vf}\times \V_{\vf}^{\vee}(1)\to \co(1)$ sends $\F^0\times \F^n$ to zero, so that $\langle \rt_F,\omega\rangle=0$ and $$V_{\vf}=\langle \rt,\pi(\omega)\rangle =\langle \rt,\omega\rangle=\langle \rt_{\QQ},\omega\rangle$$ is a period, independent from $\{\langle \ve_j,\omega\rangle \}_{j=0}^n$, and killed by $DL^{\dagger}$.
\end{proof}

\begin{example}\label{E7.1}
The (reflexive, good, tempered) Laurent polynomial $\vf=(1-x_1-x_2+x_1 x_2-x_1 x_2 x_3)\prod_{i=1}^3 (1-x_i^{-1})$ appears in the algebro-geometrization of Ap\'ery's irrationality proof for $\zeta(3)$ \cite[\S5.3]{Ke}.  Its Picard-Fuchs operator $L=D^3-t(34D^3+51D^2+27D+5)+t^2(D+1)^3$ is self-adjoint, and we have $\kappa_1=0$, $\kappa_2=-2\zeta(2)$, $\kappa_3=\tfrac{17}{6}\zeta(3)$ \cite{GZ}.  At the end of \cite{BV}, Bloch and Vlasenko ``speculate'' that the dual box extension $\V_{\vf}^{\vee}(1)|_{\fuox}$ coincides with their unipotent extension $\E$ (with $m=1$) in this case.  So Theorem \ref{T7.2} confirms this speculation.
\end{example}

\begin{rem}\label{R7.1}
If we view the $\{\ve_j\}_{j=0}^n$ as rational classes in $\MM(n+1)_{\QQ}\cong \MM_{\QQ(n+1)}$ via $(2\pi\ay)^{n+1}Q^{-1}(\cdot)\colon \MM_{\QQ}^{\vee}\to \MM_{\QQ(n+1)}$, then in the proof of Theorem \ref{T7.2} one may choose $\rt_{\QQ}=\mathsf{Q}_0^{-1}(2\pi\ay)^{n+1}\ve_{n+1}$ and extend $e_0,\ldots,e_n$ by $e_{n+1}=\mathsf{Q}_0\rt_F$.  In precise terms, the Theorem is saying that $\omega(0)=\sum_{j=0}^{n+1}(2\pi\ay)^{-j}\alpha_j^{\dagger}\tilde{\ve}_j^{\vee}(0)$ and $e_{n+1}(0)=\sum_{j=0}^{n+1}(2\pi\ay)^j\kappa_{n+1-j}^{\dagger}\tilde{\ve}_j(0)$, where the tilde means to to apply $e^{-\el(t)N_0}$.  More usefully, these can be recast as formulas for 
\begin{align*}
V_{\vf}&\equiv \tfrac{1}{\mathsf{Q}_0}\sum_{k=0}^{n+1}\alpha_{n+1-k}^{\dagger}\tfrac{\log^k(t)}{k!}\;\;\;\;\;\text{and}\\
Q(\rt)&\equiv\tfrac{1}{\mathsf{Q}_0}\textstyle{\left((2\pi\ay)^{n+1}\ve_{n+1}-\sum_{j=0}^{n+1}(2\pi\ay)^j \kappa_{n+1-j}^{\dagger}\tilde{\ve}_j\right)}
\end{align*}
modulo $\co(t\log^{n+1}t)$,\footnote{Here $\ve_{n+1}-\tilde{\ve}_{n+1}=\sum_{k=1}^{n+1}\tfrac{(-1)^k}{k!}\el^k(t)\ve_{n+1-k}$ belongs to $\M^{\vee}$; so the formula for $\rt$ makes sense.} i.e. terms which limit to zero with $t$.  In particular, we have that $V_{\vf}^{\text{an}}(0)=\tfrac{\alpha_{n+1}^{\dagger}}{\mathsf{Q}_0}$.

It was pointed out in \cite{MKtoSB} that in Example \ref{E7.1}, one can use a variant of \cite[(9.29)]{DK} to check that $\mathsf{Q}_0 V_{\vf}^{\text{an}}(0)=-\tfrac{17}{6}\zeta(3)$ (where $\mathsf{Q}_0=-\tfrac{1}{12}$).  Clearly this laborious partial confirmation of the ``speculation'' of \cite{BV} is superseded by Theorem \ref{T7.2}.
\end{rem}

\begin{example}\label{E7.2}
Writing $\vf_r(\ux):=(1+\sum_{i=1}^r x_i)(1+\sum_{i=1}^r x_i^{-1})$, the Feynman integral $I_{\r}(t):=\int_{\mathbb{R}_{\geq 0}^{\times r}} \tfrac{\mathrm{dlog}(\ux)}{1-t\vf(\ux)}$ arising from the $r$-banana graph with equal masses can be interpreted as $V_{\vf}$ (with $L=L^{\dagger}$) by the methods of \cite{BKV}.  So $I_r^{\text{an}}(0)$ is a rational multiple of the relevant $\alpha_r$, which in turn should be the top-degree coefficient of the (inverted, regularized) $\hat{\Gamma}$-class of the degree-$(1,1,\ldots,1)$ Fano hypersurface in $(\PP^1)^{\times (r+1)}$.
\end{example}

\section{Inhomogeneous equations and normal functions}\label{S8}

Recall from the proof of Theorem \ref{T5.2} that $\tilde{\Phi}(s,t)=\Phi(s,t)-\kappa(s)\phi_0(t)$ has no monodromy about $t=c$ for any fixed $s$.  Taking $s=\ell\in \ZZ_{>0}$, $$L\tilde{\Phi}(\ell,t)\;=\;L\Phi(\ell,t)-\kappa(\ell)\cancelto{\scriptstyle{0}}{L\phi_0(t)}\;=\;\ell^r t^{\ell}.$$  Moreover, $\tilde{\Phi}(\ell,t)=\sum_{k\geq 0}A_k(\ell) t^{k+\ell}-\kappa(\ell)\phi_0(t)$ is analytic at $0$. The set of solutions to $L(\cdot)=\ell^r t^{\ell}$ which are analytic at $0$ is clearly then $\{\tilde{\Phi}(\ell,t)+\mathpzc{z}\phi_0(t)\}_{\mathpzc{z}\in\CC}$, and if $\mathpzc{z}\neq 0$ these solutions have monodromy about $c$.  Since $\ell$ is a positive integer and $\Phi(\ell,t)=\sum_{k\geq 0}A_k(\ell)t^{k+\ell}$, we have $\Phi(\ell,0)=0$; and recalling in addition that $\phi_0(0)=1$ gives $\tilde{\Phi}(\ell,0)=-\kappa(\ell)$.  This proves the

\begin{thm}\label{T8.1}
Let $V^{[\ell]}(t)$ be the unique solution to the inhomogeneous equation $L(\cdot)=-t^{\ell}$ analytic at $0$ with no monodromy about $c$.  Then $\kappa(\ell)=\ell^r V^{[\ell]}(0)$.
\end{thm}

\begin{defn}\label{D8.0}
The values $\{\kappa(\ell)\}_{\ell\in\mathbb{N}}$ are called the \emph{Ap\'ery constants} of $L$.
\end{defn}

\begin{rem}\label{R8.1}
If we take $\ell\in[1,d-1]\cap \ZZ$,\footnote{Remember that $d$ is the degree of $L$ (in $t$).} then  $$b_k^{[\ell]}:=\left\{ \begin{matrix}0,& k<\ell\\ \frac{1}{\ell^r}A_{k-\ell}(\ell),& k\geq \ell \end{matrix} \right.$$ is evidently a solution to the recurrence attached to $L$.  (If $L\in \QQ[t,D]$, then the $b_k^{[\ell]}$ are also rational.)  Its generating series $\tfrac{1}{\ell^r}\Phi(\ell,t)$ and the holomorphic period $\phi_0(t)=\sum_{k\geq 0} a_k t^k$ both have monodromy about $c$, but $V^{[\ell]}(t)=\sum_{k\geq 0}(\tfrac{\kappa(\ell)}{\ell^r}a_k-b_k^{[\ell]})t^k=:\sum_{k\geq 0}v_k t^k$ does not.  So if the $a_k$ are nonzero for sufficiently large $k$, we have $$0=\lim_{k\to \infty}\frac{v_k}{a_k}\;\;\;\;\implies\;\;\;\; \frac{\kappa(\ell)}{\ell^r}=\lim_{k\to \infty}\frac{b_k^{[\ell]}}{a_k}.$$  Note that in the strong conifold monodromy case, the nonvanishing of $a_{k\gg0}$ is guaranteed by the asymptotics in the proof of Theorem \ref{T5.2}; moreover, the description of $\kappa(\ell)$ just given is consistent with Theorem \ref{T5.2}(i) since $\lim_{k\to \infty}\tfrac{a_{k-\ell}}{a_k}=c^{\ell}$ by those same asymptotics.
\end{rem}

\begin{example}\label{E8.1}
Revisiting Example \ref{E7.1} ($d=n=2$) and taking $\ell=1$, Remark \ref{R8.1} reproduces the pair of sequences $\{a_k\}=1,5,73,1445,\ldots$ and\footnote{In most of the literature, the second sequence is multiplied by $6$.  Note that any solution to the recurrence is determined by its first two terms.} $\{b_k\}:=\{b_k^{[1]}\}=0,1,2106,\tfrac{125062}{3},\ldots$ in Ap\'ery's irrationality proof for $\zeta(3)$, with limit $\kappa(1)=\lim_{k\to \infty}\tfrac{b_k}{a_k}=\tfrac{\zeta(3)}{6}$.  Though the difference between this and $\kappa_3=\tfrac{17}{6}\zeta(3)$ may seem trivial, this is an artifact of the VHS $\M$ underlying Ap\'ery possessing an ``involution'' under $t\mapsto \tfrac{1}{t}$ (cf. \cite[\S5.3]{Ke} and \cite[\S5.2]{GK}).  In general the $\{\kappa(\ell)\}$ and $\{\kappa_j\}$ describe completely different things.  The $\{\kappa(\ell)\}$ are closely related, as we shall see, to \emph{special values at $0$ of normal functions} nonsingular at $0$, as well as to Galkin's Ap\'ery constants of Fano varieties \cite{Ga}.  The $\{\kappa_j\}$ are extension-class invariants of the LMHS at $0$ of the unipotent extensions of \S\ref{S7} (but \emph{cannot} be evaluated as the \emph{limit of an extension} at $0$), and are closely tied to the Gamma constants of Fano varieties \cite{GGI,GZ}.
\end{example}

We shall conclude this article by saying something about these special values of normal functions.  Given a polarized $\QQ$-VHS $\H$ on $U$ (of negative weight), there are \emph{singularity invariants} $$\mathrm{sing}_{\sigma}\colon \mathrm{ANF}(\H)\to \mathrm{Hom}_{\text{MHS}}(\QQ(0),(\psi_{\sigma}\H)_{T_{\sigma}}(-1))$$ attached to each $\sigma\in\Sigma$ \cite[\S2.12]{KP}.  (This is essentially the restriction map $H^1(U,\HH)\to H^1(\Delta_{\sigma}^{\times},\HH)$ applied to Hodge-$(0,0)$ classes, where $\Delta_{\sigma}^{\times}$ is a punctured disk about $\sigma$.)  One says that $\V\in\mathrm{ANF}(\H)$ is \emph{singular} at $\sigma$ if $\mathrm{sing}_{\sigma}(\V)\neq 0$.  Writing $h:=\deg(\F_e^n)$ for the degree of the Hodge line bundle, we have the

\begin{lem}[\cite{GK}] \label{L8.1}
Given $a\in\ZZ_{>0}$ and $\nf\in\mathrm{ANF}(\M(a))\setminus \{0\}$ nonsingular away from $\infty$, let $\tilde{\nf}$ be a \textup{(}multivalued\textup{)} lift to $\M$ of the associated section of $J(\M(a))$, and $\v(t):=Q(\tilde{\nf},\mu)$ the resulting \textup{(}multivalued, holomorphic\textup{)} truncated HNF on $U$.  Then $L\v$ is a nonzero polynomial in $t$ vanishing at $t=0$, of degree $\leq d-h$.
\end{lem}

\begin{rem}\label{R8.2}
(i) If $\nf$ is nonsingular away from $0$ instead, the result in \cite{GK} (which works in greater generality than our setting) says that $\deg(L\v)\leq d-h-1$ if $T_{\infty}$ is unipotent and $\leq d-h$ otherwise.  (However, $L\v$ need not vanish at $0$.)  The box extensions $\V_{\vf}$ from \S\ref{S7} are of this type, with $h=1$, and the proof of Theorem \ref{T7.2} shows --- writing $\v_{\vf}:=Q(\rt,\mu)=pV_{\vf}$ --- that $L\v_{\vf}=pLV_{\vf}=\tfrac{(-1)^n}{\mathsf{Q}_0}p$.  So in the setting of Definition \ref{D7.1}, we get that $\deg(p)\leq d-2$ resp. $d-1$ (depending on $T_{\infty}$).

(ii) Continuing with this setting, there are immediate consequences for the lowest degrees. Clearly if $d=1$ then $\deg(p)=0$, $T_{\infty}$ is non-unipotent, and $L^{\dagger}=L$. In fact, if $d=2$ we also have $L^{\dagger}=L$. To see this, write $c'$ for the second root of $q_0$. If $c'\in U$, then $\mu,\partial\mu,\ldots,\partial^n \mu$ do not span $\M_{c'}$ ($\partial^{n+1}\mu$ is not an $\co_{c'}$-linear combination of them); so there is a gap in the Kodaira-Spencer maps and $Y(c')=0$. A similar argument shows $Y(c')=0$ if $c'=c$ ($\mathrm{ord}_c(q_0)=2$). Either way, $p$ has (at least) a double zero at $c'$, contradicting $\deg(p)\leq 1$. So $\Sigma^{\times}=\{c,c'\}$ and $1=\mathrm{ord}_{c'}(q_0)\geq \rk(T_{c'}-I)$ forces conifold monodromy at $c'$; moreover, no Kodaira-Spencer maps vanish anywhere\footnote{Any vanishing of a K-S map \emph{at} a conifold monodromy point is away from the center, hence duplicated by the self-duality; so $Y$ has odd order. Any vanishing of a K-S map (equivalently, of $Y$) on $U$ makes $q_0$ vanish.} on $\CC^{\times}$.  So $q_0 Y$ is constant $\implies$ $p\equiv 1$ $\implies$ $L^{\dagger}=L$. (In contrast, if $d=3$ there are examples like the family generated by $\vf=\tfrac{(1+x_1+x_2^2)^2}{x_1 x_2}-8$, with $\Sigma^{\times}=\{-\tfrac{1}{16},-\tfrac{1}{8}\}$, $q_0=(1+16t)(1+8t)^2$, and $p=1+8t$. The trouble is the $I_0^*$ fiber at $t=-\tfrac{1}{8}$.)

(iii) If $\nf$ is nonsingular everywhere (and nontrivial), then $\M(a)$ must have weight $-1$ ($\iff$ $n$ odd and $a=\tfrac{n+1}{2}$), which corresponds to ``classical'' normal functions. In this case, $deg(L\v)\leq d-h-1$ resp. $d-h$ \emph{and} $L\v$ vanishes at $0$.  See Example \ref{E8.2}(b) below.
\end{rem}

We shall use the Lemma to prove a result which, together with Theorem \ref{T8.1}, produces an interpretation of (some) $\kappa(\ell)$'s as special values (at $t=0$) of normal functions.

\begin{thm}\label{T8.2}
Suppose there exists an embedding $$\vt\colon \QQ(-a)\hookrightarrow \IH^1(\mathbb{A}^1,\MM),$$ where $\mathbb{A}^1=\PP^1\setminus \{\infty\}$.  Then there is a normal function\footnote{See the proof for the precise correspondence with $\vt$.} $$\nf_{\vt}\in \mathrm{ANF}(\M(a))\setminus \{0\},$$ with $\v_{\vt}:=Q(\tilde{\nf}_{\vt},\mu)$ satisfying $L\v_{\vt}(t)=tP_{\vt}(t)$, where $P_{\vt}\in\CC[t]\setminus\{0\}$ has $\deg(P_{\vt})\leq d-h-1$.  The lift $\tilde{\nf}_{\vt}$ can be chosen so that $\v_{\vt}$ is analytic on a disk of radius $>|c|$ about the origin.
\end{thm}

\begin{proof}
Recall that $\M$ is a summand of the $n^{\text{th}}$ cohomology of some $f_U\colon \X_U\to U$, or more precisely of its quotient $H_{\text{var}}^n$ by the fixed part $H^n_{\text{fix}}=H^0(U,R^n (f_U)_* \QQ_{\X_U})$. Let $\X\supset \X_U$ be our smooth compactification, and consider the extension in $\mathrm{AVMHS}(U)$ with fibers\small $$0\to H^n_{\text{var}}(X_t)\to H^{n+1}(\X\setminus X_{\infty},X_t)\to \ker \{H^{n+1}(\X\setminus X_{\infty})\to H^{n+1}_{\text{fix}}\}\to 0.$$\normalsize Pushing forward by $H^n_{\text{var}}\twoheadrightarrow \M$ on the left and pulling back by the composition of $\vt$ with the inclusion of $\IH^1(\mathbb{A}^1,\MM)$ on the right, we get an element $\nf_{\vt}\in \mathrm{Ext}^1_{\mathrm{AVMHS}(U)}(\QQ(-a),\M)\cong \mathrm{Ext}^1_{\mathrm{AVMHS}(U)}(\QQ(0),\M(a))$. Its topological invariant $[\nf_{\vt}]\in \mathrm{Hom}_{\text{MHS}}(\QQ(0),H^1(U,\MM)(a))$ is tautologically the (nonzero) image of $1$ under $\QQ(0)\overset{\vt}{\hookrightarrow}\IH^1(\mathbb{A}^1,\MM(a))\hookrightarrow H^1(U,\MM(a))$. In particular, it has no singularities on $\mathbb{A}^1$.  Apply Lemma \ref{L8.1} to this $\nf_{\vt}$.

It remains to check existence of a lift $\tilde{\nf}_{\vt}$ with no monodromy on $\fux$. This boils down to whether $[\nf_{\vt}]$ restricts to zero in $H^1(\fux,\MM)$. Writing $\bar{\fu}=\fux\cup\{0,c\}$, $[\nf_{\vt}]|_{\fux}$ clearly lies in the image of $\IH^1(\bar{\fu},\MM)$.  But in the Mayer-Vietoris sequence $$\MM_p^{T_0}\oplus \MM_p^{T_c}\to \MM_p \to \IH^1(\bar{\fu},\MM)\to \IH^1(\fuo,\MM)\oplus \IH^1(\fuc,\MM)$$ the first arrow is surjective (replace $\MM$ by $\MM^{\vee}$ and argue that $(\MM^{\vee}_p)^{T_0}$ contains $\ve_0$ and $(\MM_p^{\vee})^{T_c}$ contains $\ve_1,\ldots,\ve_n$) and the final term is zero; so we are done.
\end{proof}

\begin{rem}\label{R8.3}
The existence of the single-valued lift on $\fux$ is made out to be a harder result in more special cases in \cite{DK,BKV,Ke}; but this is because for the applications in those works, an exact identification of the current representing the lift was required.
\end{rem}

\begin{defn}\label{D8.1}
The extension of $\QQ$-VMHS $$0\to \M\to \nf_{\vt}\to \QQ(-a)\to 0$$ corresponding to the normal function in Theorem \ref{T8.2} is called an \emph{Ap\'ery extension}.
\end{defn}

The Beilinson-Hodge Conjecture predicts the existence of a cycle $\mathfrak{Z}_{\vt}\in \mathrm{CH}^a(\X\setminus X_{\infty},2a-n-1)_{\QQ}$ giving rise to $\nf_{\vt}$.  When this exists, $\v_{\vt}(0)$ can be computed (up to $\QQ(a)$) as follows: first, $\mathfrak{Z}_{\vt}(0):=\imath^*_{X_0}\mathfrak{Z}_{\vt}\in H^{n+1}_{\mathrm{M}}(X_0,\QQ(a))$ has $\mathrm{AJ}_{X_0}(\mathfrak{Z}_{\vt}(0))\in \mathrm{Ext}^1_{\text{MHS}}(\QQ(0),H^n(X_0,\QQ(a)))$.  Next, the composition $\QQ(0)\cong (\psi_0\M^{\vee})_{T_0}\hookrightarrow H^n_{\lim}(X_t)_{T_0}(n)\overset{\mathtt{sp}}{\to}H_n(X_0)$ of MHS-morphisms sends $1\mapsto Q(\mu_0)\mapsto \mathrm{Res}_{X_0}(\tfrac{\mathrm{dlog}(\ux)}{(2\pi\ay)^n})=:\mu_{X_0}$; and so pairing with $\mu_{X_0}$ induces $H^n(X_0,\QQ)\twoheadrightarrow \QQ(0)$. By \cite[Cor. 5.3]{7K} we therefore have
\begin{align*}
\v_{\vt}(0)&\equiv \lim_{t\to 0}Q(\mathrm{AJ}_{X_t}(\mathfrak{Z}_{\vt}(t)),\mu_t)\\
&\equiv \langle \mathrm{AJ}_{X_0}(\mathfrak{Z}_{\vt}(0),\mu_{X_0}\rangle \in \CC/\QQ(a)\cong \mathrm{Ext}^1_{\text{MHS}}(\QQ(0),\QQ(a).
\end{align*}
In this scenario, the second line typically factors through the ``Borel'' regulator $H^1_{\mathrm{M}}(\text{Spec}(K),\QQ(a))\to \CC/\QQ(a)$, with $K$ the field of definition of $\mathfrak{Z}_{\vt}$.  When $K=\QQ$, one then has $\v_{\vt}(0)\in\QQ\zeta(a)$.  Note that for families of $K3$ surfaces ($n=2$), there are two possibilities: $a=3$ and $a=2$. Both do occur \cite{GK}.  Similarly, for elliptic curves ($n=1$), Example \ref{E8.2} below shows that both $a=2$ and $a=1$ happen.

Putting everything together, \emph{provided one can find enough embeddings} $\vt$, and either assuming the BHC or constructing the cycles, one would obtain that:
\begin{itemize}[leftmargin=0.5cm]
\item $\kappa(1),\ldots,\kappa(d-h)$ are of the form $\v_{\vt}(0)\underset{\QQ(0)}{\equiv}\langle\mathrm{AJ}_{X_0}(\mathfrak{Z}_{\vt}(0)),\mu_{X_0}\rangle$, hence are periods; and
\item with an assumption on the field of definition, they are actually rational multiples of Riemann zeta values.
\end{itemize}
However, we caution the reader that there are several obstacles to the existence of such embeddings (especially multiple, independent ones), the first of which is that $\IH^1(\mathbb{A}^1,\MM)$ may not be Hodge-Tate.  
Even if it is, it can possess nontrivial extension classes which ``obstruct'' such embeddings (which are after all Hodge classes), meaning that one must consider biextensions of VMHS on $U$; though in that case it is likely that the resulting $\kappa(j)$'s can still be analyzed in terms of (higher) cycles on Zariski-open subsets of the fibers.  Our assumption that $\M$ be of type $(1,1,\ldots,1)$ also imposes severe limitations: if $n$ is even, then there can be at most one\footnote{This follows from the proof of Theorem \ref{T8.3} below.} Hodge class in $\IH^1(\mathbb{A}^1,\MM)$; but this just means that a more general study is in order.

We finish with one (still fairly broad) case where we only want \emph{one} embedding, and that embedding fortunately \emph{must} exist:

\begin{thm}\label{T8.3}
Assume that $\M$ arises from a good, reflexive, tempered Laurent polynomial $\vf$, and that $d=2$.  Then we have an isomorphism $\IH^1(\mathbb{A}^1,\MM)\overset{\cong}{\underset{\vt}{\leftarrow}}\QQ(-a)$ for some $a\in[\tfrac{n+1}{2},n+1]\cap\ZZ$.  The resulting admissible normal function satisfies $L\v_{\vt}=-kt$ for some $k\in\bar{\QQ}^{\times}$, and $\kappa(1)=\tfrac{1}{k}\v_{\vt}(0)$.
\end{thm}

\begin{proof}
Note that by Remark \ref{R8.2}(ii), $\Sigma^{\times}$ comprises two conifold points (and also $L^{\dagger}=L$).  By Euler-Poincar\'e, the rank of $\IH^1(\mathbb{A}^1,\MM)$ is given by $\sum_{\sigma\in\Sigma\setminus\{\infty\}}\rk(T_{\sigma}-I)-r\chi(\mathbb{A}^1)=n+1+1-(n+1)=1.$ So one of the end terms in the exact sequence of MHS $$0\to \IH^1(\PP^1,\MM)\to \IH^1(\mathbb{A}^1,\MM)\to(\psi_{\infty}\M)_{T_{\infty}}(-1)\to 0$$ is zero, and the other has rank one.  (Applying E-P to the first term, either $\rk(T_{\infty}-I)=n$ and the first term vanishes, or it $=n+1$ and the last term vanishes.)  A rank-one MHS is of the form $\QQ(-a)$; and the first term can only have weight $n+1$, while the last term can have weights between $n+2$ and $2n+2$. 
\end{proof}

\begin{rem}\label{R8.4}
The examples in \cite{Go} are of this type, and the corresponding (higher) normal functions are constructed explicitly in \cite{GK}.  In the $V_{18}$ case, as noticed by \cite{dS}, we have $k\notin \QQ$ (in fact $k=\sqrt{-3}$, and $\v_{\vt}(0)\in(2\pi\ay)^3 \QQ$); though the family is defined over $\QQ$, the normalization of $X_{\infty}$ (and consequently $\mathfrak{Z}_{\vt}$) is only defined over $\QQ(\sqrt{-3})$.
\end{rem}

\begin{example}\label{E8.2}
For $n=1$ (and $d=2$), we demonstrate the two possibilities in Theorem \ref{T8.3}:\vspace{2mm}

\noindent {\bf (a)} $\underline{\IH^1(\PP^1,\MM)=\{0\}}$:  $\vf=(1-x_1^{-1})(1-x_2^{-1})(1-x_1-x_2)$ yields the ``little Ap\'ery'' family associated with irrationality of $\v_{\vt}(0)=\zeta(2)$, where $\mathfrak{Z}_{\vt}\in \mathrm{CH}^2(\X\setminus X_{\infty},2)$ is obtained by pulling the box cycle $\bx$ back along the involution $(x_1,x_2,t)\mapsto (\tfrac{x_1}{x_1-1},\tfrac{1-x_2}{1-x_1-x_2},-\tfrac{1}{t})$ \cite[\S5.2]{Ke}.  In direct analogy to \cite[\S5.2]{GK}, one can show that \small $$\v_{\vt}(t)=\int_{\RR_{\leq 0}^{\times 2}}\frac{\mathrm{dlog}(\ux)}{t+\vf(\ux)}=\sum_{j\geq 0} (-t)^j \int_{\RR_{\leq 0}^{\times 2}}\frac{\mathrm{dlog}(\ux)}{\vf(\ux)^{j+1}}=\zeta(2)+(3\zeta(2)-5)t+\cdots.$$\normalsize Applying $L=D^2-t(11D^2+11D+3)-t^2(D+1)^2$ and invoking Theorem \ref{T8.3}, we find $k=5$ hence $\kappa(1)=\tfrac{\zeta(2)}{5}$.\vspace{2mm}

\noindent{\bf (b)} $\underline{\IH^1(\PP^1,\MM)\neq \{0\}}$:  $\vf=x_1^{-1}x_2^{-1}(1+x_1+x_2+x_2^2)^2$ yields a family with singular fibers of types $\mathrm{I}_4,\mathrm{I}_1,\mathrm{I}_1,\mathrm{I}_0^*$ at $0,\tfrac{1}{12},-\tfrac{1}{4},\infty$ respectively. It has a nontorsion\footnote{Observe that $x_1=3(u-1)^{-4}(u-\ay\sqrt{3})^2(u-\tfrac{\ay}{\sqrt{3}})^2$, $x_2=(u-1)^{-2}(u+1)^2$ yields a normalization $\PP^1\to X_{\frac{1}{12}}$ sending $u=0,\infty$ to the node $(3,1)$. The preimage of the cycle is $[\tfrac{\ay}{\sqrt{3}}]-[\ay\sqrt{3}]$, and $\tfrac{\ay/\sqrt{3}}{\ay\sqrt{3}}=\tfrac{1}{3}\in\CC^{\times}$ has infinite order.} section given by $\mathfrak{Z}_{\vt}=[(0,\zeta_3)]-[(0,-\zeta_3^2)]\in\mathrm{CH}^1(\X)$, where $\zeta_3:=e^{\frac{2\pi\ay}{3}}$. The Abel-Jacobi map yields 
\begin{align*}
\v_{\vt}(t)&=\int_{(0,\zeta^2_3)}^{(0,\zeta_3)}\mu=\frac{1}{2\pi\ay}\int_{\zeta_3^2}^{\zeta_3}\oint_{|x_1|=\epsilon}\frac{dx_1/x_1}{1-t\vf(\ux)}\frac{dx_2}{x_2}=\sum_{j\geq 0} t^j \int_{-\ay}^{\ay}[\vf^k]_{x_1}\frac{dx_2}{x_2}\\
&=-\tfrac{2}{3}\pi\ay+(4\sqrt{3}\ay-\tfrac{4}{3}\pi\ay)t+(18\sqrt{3}\ay-12\pi\ay) t^2+\cdots\,,
\end{align*}
where $[\vf^k]_{x_1}$ means  terms constant in $x_1$. Applying $L=(1-8t-48t^2)D^2-(8t+96t^2)D-(2t+36t^2)$ and invoking Theorem \ref{T8.3} once more, we find $k=-4\sqrt{3}\ay$ and $\kappa(1)=\tfrac{\pi}{6\sqrt{3}}$.\vspace{2mm}

\noindent{\bf (c)}  The simplest example of what we mean by an ``obstruction'' occurs for $n=1$ and $d=3$, for the polynomial $\vf=x_1^{-1}x_2^{-1}(1+x_1+x_2^2)^2-8$ from the end of Remark \ref{R8.2}(ii) (with an $\mathrm{I}_1$ at $\infty$). As in (b), there is a nontorsion section $\mathfrak{Z}=[(0,\ay)]-[(0,-\ay)]\in \mathrm{CH}^1(\X)$, which limits in particular to $\tfrac{\ay(\sqrt{2}-1)}{\ay(\sqrt{2}+1)}=3-2\sqrt{2}\in \CC^{\times}$ in the group law on $X_{\infty}^{\text{sm}}$.\footnote{Normalize $X_{\infty}$ by $x_1=2(u-1)^{-4}(u-\ay(\sqrt{2}+1))^2(u-\ay(\sqrt{2}-1))^2$, $x_2=(u-1)^{-2}(u+1)^2$; the preimage of $\mathfrak{Z}$ is $[(\ay(\sqrt{2}-1)]-[(\ay(\sqrt{2}+1)]$.}
The difference is that in this case $\IH^1(\mathbb{A}^1,\MM)$ has rank $2$, and is a (nonsplit) extension of $(\psi_{\infty}\M)_{T_{\infty}}(-1)\cong \QQ(-2)$ by $\IH^1(\PP^1,\MM)\cong \QQ(-1)$ with class $\log(3-2\sqrt{2})\in \CC/\QQ(1)\cong \mathrm{Ext}^1_{\text{MHS}}(\QQ(-2),\QQ(-1))$.  So there is no morphism $\QQ(-2)\hookrightarrow \IH^1(\mathbb{A}^1,\MM)$, and one must deal with biextensions.  This still may be treated via a higher cycle, but this cycle lives in $\mathrm{CH}^2(\X\setminus \{X_{\infty}\cup|\mathfrak{Z}|\},2)$ and does not lift to $\mathrm{CH}^2(\X\setminus X_{\infty},2)$.
\end{example}

\begin{rem}\label{R8.5}
We have argued above that $\kappa(1),\ldots,\kappa(d-1)$ are interesting invariants of $\M$ related to algebraic cycles; the natural reaction is to wonder if $\kappa(d)$, $\kappa(d+1)$, etc. are similarly interesting.  In fact, to expand on \cite[Rem. 32]{BV} a bit, they are not: taking $L\in K[t,D]$, they are always contained in $K[\kappa(1),\ldots,\kappa(d-1)]$ in view of Corollary \ref{C6.0}. For example, if $L^{\dagger}=L$ ($\implies$ $Q_j=P_j$) then $$\kappa(d)=\frac{-d^r}{P_d(-d)}\sum_{j=0}^{d-1} j^{-r} P_j(-j)\kappa(j),$$ where ``$0^{-r}P_0(0)$'' is to be read as $\lim_{s\to 0}s^{-r}P_0(-s)=\lim_{s\to 0}\tfrac{(-s)^r}{s^r}=(-1)^r$. So in the Ap\'ery $\zeta(3)$ case (Examples \ref{E7.1} and \ref{E8.1}), where $\kappa(1)=\tfrac{\zeta(3)}{6}$ (and $\kappa(0)=1$), we find $\kappa(2)=-8+\tfrac{5}{6}\zeta(3)$; one can also show (in the notation of Theorem \ref{T8.1}) that the solutions of the inhomogeneous equations satisfy $V^{[2]}(t)=-1-\tfrac{35}{48}\zeta(3)A(t)+5V^{[1]}(t)$.
\end{rem}

\curraddr{${}$\\
\noun{Dept. of Mathematics and Statistics, Campus Box 1146}\\
\noun{Washington University in St. Louis}\\
\noun{St. Louis, MO} \noun{63130, USA}}

\email{${}$\\
\emph{e-mail}: matkerr@wustl.edu}
\end{document}